\documentclass[reqno, amsthm, 11pt]{amsart}

\usepackage[dvipsnames]{xcolor}

\usepackage{tikz}
\usepackage{tikz-cd}
\usetikzlibrary{shapes,arrows,cd}

\usepackage{mathrsfs,amssymb}

\usepackage{todonotes}
\usepackage{pdfsync}
\usepackage{enumitem}

\usepackage{rotating}

\usepackage{graphicx}

\def\Sclo{\overrightarrow{\Sigma^*}}
\def\U{\mathcal{U}}
\def\A{\mathcal{A}}

\usepackage{amsfonts} 
\usepackage{amsmath}
\usepackage{amssymb}
\usepackage{latexsym}
\usepackage{mathrsfs}
\usepackage{graphicx}

\usepackage{epic,eepic,ecltree} 

\DeclareMathOperator\pref{pref} 
\DeclareMathOperator\suff{suff} 
\DeclareMathOperator\red{red}

\newcommand{\Mpres}[2]{\operatorname{Mon}\bigl\langle #1\:|\:#2 \bigr\rangle}
\newcommand{\Mgen}[1]{\operatorname{Mon}\bigl\langle #1 \bigr\rangle}
\newcommand{\Ggen}[1]{\operatorname{Gp}\bigl\langle #1 \bigr\rangle}

\newcommand{\Gpres}[2]{\operatorname{Gp}\bigl\langle #1\:|\:#2 \bigr\rangle}
\newcommand{\Ipres}[2]{\operatorname{Inv}\bigl\langle #1\:|\:#2 \bigr\rangle}
\newcommand{\Igen}[1]{\operatorname{Inv}\bigl\langle #1 \bigr\rangle}

\input xy
\xyoption{all}

\newcommand{\N}{\mathbb{N}}

\usepackage{xy} 

\newcommand{\dedge}[1]{\ar@{--}[#1]}

\newcommand{\edge}[1]{\ar@{-}[#1]}

\newcommand{\lulab}[1]{\ar@{}[l]_<<{#1}}

\newcommand{\rulab}[1]{\ar@{}[r]^<<{#1}}

\newcommand{\ldlab}[1]{\ar@{}[l]^<<{#1}}

\newcommand{\rdlab}[1]{\ar@{}[r]_<<{#1}}

 \newcommand{\OHare}{\mathcal{O}}

\newtheorem{thm}{Theorem}[section]

\newtheorem{cor}[thm]{Corollary}
\newtheorem{prop}[thm]{Proposition}
\newtheorem{lem}[thm]{Lemma}

\theoremstyle{definition}
\newtheorem{defn}[thm]{Definition}

\newtheorem{remark}[thm]{Remark}

\newtheorem{question}[thm]{Question}

\newtheorem{example}[thm]{Example}

\newcommand{\lb}{\langle}

\newcommand{\rb}{\rangle}

\newcommand{\restr}{\!\restriction}

\usepackage[active]{srcltx}

\usepackage{verbatim}

\usepackage{enumitem}

\begin{document}


\title[Groups of units of one-relator inverse monoids]{
On groups of units of special and \\ one-relator inverse monoids
}


\author{ROBERT D. GRAY}
\email{Robert.D.Gray@uea.ac.uk}

\address{School of Mathematics, University of East Anglia, Norwich NR4 7TJ, England, UK}

\author{NIK RU\v{S}KUC}
\email{nik.ruskuc@st-andrews.ac.uk}

\address{School of Mathematics and Statistics, University of St Andrews, St Andrews KY16 9SS, Scotland, UK.}



\keywords{One-relator monoid, 
one-relator group, 
inverse monoid, 
special inverse monoid, 
units, right units, coherence.
\newline
\indent This research of R.\ D. Gray was supported by the EPSRC-funded projects 
EP/N033353/1 ``Special inverse monoids: subgroups, structure, geometry, rewriting systems and the word problem'', and 
EP/V032003/1 `Algorithmic, topological and geometric aspects of infinite groups, monoids and inverse semigroups'. 
}


\begin{abstract}
We investigate the groups of units of one-relator and special inverse monoids. These are inverse monoids which are defined by presentations where all the defining relations are of the form $r=1$. We develop new approaches for finding presentations for the group of units of a special inverse monoid, and apply these methods to give conditions under which the group admits a presentation with the same number of defining relations as the monoid. In particular our results give sufficient conditions for the group of units of a one-relator inverse monoid to be a one-relator group. When these conditions are satisfied these results give inverse semigroup theoretic analogues of classical results of Adjan for one-relator monoids, and Makanin for special monoids. In contrast, we show that in general these classical results do not hold for one-relator and special inverse monoids. In particular, we show that there exists a one-relator special inverse monoid whose group of units is not a one-relator group (with respect to any generating set), and we show that there exists a finitely presented special inverse monoid whose group of units is not finitely presented.
\end{abstract}


\maketitle

\section{Introduction and summary of results}
\label{sec:intro}

The purpose of this paper is to undertake a systematic investigation into the groups of units of one-relator and special inverse monoids.
Throughout there is a particular emphasis on comparing and contrasting these groups with groups of units of special monoids,
after the work of Adjan, Makanin, Lallement and Zhang, 
as well as with one-relator groups themselves; see 
\cite{adjan,Lallement1974,makanin,zhang:rewriting}.

The word problem for semigroups defined by a single relation is one of the oldest, most famous and most elementary to state open problems in the broad area of combinatorial algebra.
It asks whether for every monoid of the form $M=\Mpres{A}{u=v}$, i.e. defined by generators $A$ and a single defining relation $u=v$, there is an algorithm which for any two words $w_1,w_2\in A^\ast$ decides whether or not $w_1=w_2$ in $M$.
In his seminal work on the subject in the 1960s and 70s, Adjan proved, among other things, the following:
\begin{enumerate}[leftmargin=10mm,itemsep=1mm,label=\textup{(A\arabic*)}]
\item
\label{it:A1}
The word problem is soluble for every \emph{one-relator} special monoid, i.e. monoid defined by $M=\Mpres{A}{r=1}$;
see \cite{adjan}.
\item
\label{it:A2}
The word problem is soluble for every monoid of the form $M=\Mpres{A}{u=v}$, where $u,v$ are non-empty words whose first 
letters are distinct and last letters are distinct; see \cite{adjan}.
\item
\label{it:A3}
If the word problem is soluble for all monoids of the form $M=\Mpres{A}{au=av}$ where $a\in A$, and the final letters of $u$ and $v$ are distinct, then the word problem for \emph{all} one relation monoids is soluble; 
with Oganesjan \cite{adog78}.
\end{enumerate}

Adjan's proof of \ref{it:A1} and some subsequent developments are of particular relevance for our work, and we outline them here.
The key feature is a focus on the group of units. Specifically, suppose
 $r$ is decomposed as $r\equiv r_1\dots r_k$, where $r_i$ are non-empty and invertible and of shortest possible length subject to this requirement; we will call 
 such $r_i$ the \emph{minimal invertible pieces} of $r$. Then the following holds:

\begin{thm}[Adjan \cite{adjan}]
\label{thm:MonPiecesGenerate}
If $M=\Mpres{A}{r=1}$, and if $r=r_1\dots r_k$ is the decomposition into minimal invertible pieces, then the group of units
$U=U(M)$ of $M$ is generated by $\{r_1,\dots,r_k\}$. Furthermore, if $B$ is an alphabet in one-one correspondence
$r_i\mapsto \overline{r}_i$ with the set $\{r_1,\dots, r_k\}$, then the group presentation
$\Gpres{B}{\overline{r}_1\dots \overline{r}_n=1}$ defines $U$ with respect to this generating set.
\end{thm}

It follows then that the group of units of $M$ is a one-relator group, and hence has a soluble word problem by a classical result of Magnus; see \cite{lyndonschupp, mks}.
It can be shown that the submonoid of right units $R$ of $M$ is the free product of $U$ and a free monoid of finite rank, and hence it is also finitely presented and has a soluble word problem. A further normal form theorem can be proved that reduces the word problem for  
the entire monoid $M$ to those of $U$ and $R$, and thus $M$ has a soluble word problem too.

Makanin \cite{makanin} generalises this approach to special monoids, i.e. monoids defined by presentations of the form
$M=\Mpres{A}{r_i=1\ (i\in I)}$. This time we decompose each $r_i$ into minimal invertible pieces $r_i\equiv r_{i1}\dots r_{ik_i}$.
Furthermore, we assign to each piece $r_{ij}$ a new letter $\overline{r}_{ij}$, but this time in such a way that
$\overline{r}_{ij}=\overline{r}_{lm}$ whenever $r_{ij}=r_{lm}$ \emph{in the monoid} $M$.
Collecting these new letters into the alphabet $B$ we have:

\begin{thm}[Makanin \cite{makanin}]
\label{thm:SpecialMakanin}
The group of units of $M=\Mpres{A}{r_i=1\ (i\in I)}$ is generated by the set
$\{ r_{ij}\::\: i\in I,\ 1\leq j\leq k_i\}$, and, in terms of these generators, is defined by the presentation
$\Gpres{B}{\overline{r}_{i1}\dots\overline{r}_{ik_i}=1\ (i\in I)}$.
\end{thm}

It follows that the group of units of a special monoid can be defined by a presentation with no more defining relations than the original presentation for $M$. 

A key property of the decompositions into minimal invertible pieces is that they do not overlap with each other. This in fact yields an algorithm for computing the decomposition into minimal invertible pieces for one-relator monoids; see Lallement \cite{Lallement1974}.
The algorithm proceeds by finding the overlaps of $r$ with itself, forming certain cyclic conjugates of $r$, and then repeating this process.
This algorithm also makes sense in the context of special monoids too, and it computes some decomposition of the relators into invertible pieces. It turns out, however, that these pieces need not be minimal in general (see Example~\ref{ex:AdjanFailsForNonOR}).
All the results surveyed above concerning special monoids were revisited and simplified by Zhang \cite{zhang:rewriting,zhang:short} using the methodology of string rewriting systems.

In a separate strand, the 1990s and early 2000s saw  a dynamic development of the theory of presentations for inverse monoids. The catalyst for this was Stephen's discovery \cite{stephen90} of an algorithmic procedure, akin to the Todd--Coxeter algorithm from combinatorial group theory, which computes the so called Sch\"{u}tzenberger graph of an inverse monoid $M=\Ipres{A}{R}$ based at an arbitrary word $w\in \overline{A}^\ast$. 
Linking this development with \ref{it:A3}, Ivanov, Margolis and Meakin \cite{imm01} showed that
if all one-relator inverse monoids $\Ipres{A}{r=1}$, 
where $r$ is a reduced word, 
have soluble word problems then this would imply the same for the 
monoids of the form $M=\Mpres{A}{au=av}$, and hence for all one-relation monoids.
Relating to this, it was recently shown in \cite{gray:undecidable} that, in general, not 
every one-relator inverse monoid $\Ipres{A}{r=1}$ has soluble word problem. 
However the question of solubility of the word problem in the case that $r$ is a reduced word is still open.  
In their paper \cite{imm01} Ivanov, Margolis and Meakin  
also essentially proved the generation part of the analogue of Theorem \ref{thm:SpecialMakanin} for special inverse monoids:

\begin{thm}[Ivanov, Margolis, Meakin {\cite[Proposition 4.2]{imm01}}]
\label{IMM:gens}
Suppose
\[
M=\Ipres{A}{r_i=1 \ (i\in I)}
\]
 is a special inverse monoid, and that each relator $r_i$ is decomposed into
minimal invertible pieces $r_i\equiv r_{i1}\dots r_{ik_i}$. Then the set $\{ r_{ij}\::\: i\in I,\ 1\leq j\leq k_i\}$ generates the group of units of $M$.
\end{thm}

\begin{proof}
The authors state their theorem under the assumption that all $r_i$ are cyclically reduced words. However, on inspection, this condition is not used in the proof, and hence their proof in fact establishes the theorem as stated here.
\end{proof}

However, there has been no further development along the Adjan/Makanin/Zhang lines, toward establishing presentation properties of the group of units. In fact, partly to hint at the difficulties that such an attempt would entail,
Margolis, Meakin and Stephen \cite{OHarePaper} introduce the following specific one-relator monoid
\[
\OHare=\Ipres{a,b,c,d}{abcdacdadabbcdacd = 1}=\Ipres{a,b,c,d}{r = 1},
\]
which has since come to be known as the O'Hare monoid. It is defined by a single, positive relator $r$ (and hence, in particular, it is $E$-unitary). 
The relator has no overlaps with itself, so the Adjan overlap algorithm would terminate instantly, and return $r$ as a single invertible piece of itself. However, using van Kampen diagrams and Stephen's procedure, 
Margolis and Meakin succeed in showing that $r$ in fact has a finer decomposition, namely 
\[
r\equiv abcd\cdot 
acd\cdot
ad\cdot
abbcd\cdot
acd.
\]
They provide no further information about the group of units.

Motivated by the developments described above, in this paper we investigate the extent to which the Adjan/Makanin/Zhang results for groups of units, and monoids of right units, of one-relator and special monoids are also valid for one-relator and special inverse monoids. 
In a nutshell, we will see that they do often, but not always. Exploration of when they do will lead us to a general theorem, and several applications in concrete situations with extra assumptions. On the negative side, we will see that the results may fail to generalise to inverse semigroups for each of two possible reasons: the group defined by the natural presentation on the minimal invertible pieces may not be isomorphic to the group of units, and there are inherent structural difficulties in attempting to compute the minimal pieces.

To elaborate a bit further, we consider the question of whether the group of units of a one-relator inverse monoid is a one-relator group.
In the positive direction, in Section \ref{sec:Makanin} we prove the following general sufficient condition for this to be the case:

\theoremstyle{plain}
\newtheorem*{thmUcopy}{Theorem \ref{thm:Usandwiched}}

\newcommand{\thmUcopyText}{
Let
$M= \Ipres{A}{r_i=1 \; (i \in I)}$
 be a special inverse monoid, let
$r_{i}\equiv r_{i 1}r_{i 2}\ldots r_{i k_i} $
be a factorisation into invertible pieces for $i\in I$, and
let $H$ be the subgroup of the free group
$F_A$
generated by all the pieces
$r_{ij} \; ( i\in I, 1\leq j\leq k_{i} ) $.
Fix an isomorphism $\phi :F_Y \rightarrow H$ and set
\[
K = \Gpres{Y}{
\phi^{-1}(r_{i 1}) \phi^{-1}(r_{i 2})
\ldots 
\phi^{-1}(r_{i k_i}) = 1 \; (i \in I) 
}.
\]
If $\phi$ induces an embedding of the group $ K$ into the group $\Gpres{A}{r_i=1 \; (i \in I)}$
defined by the same presentation as $M$,
then $\phi $ induces an isomorphism between $K$
and the subgroup of the monoid $M$ generated by the pieces.
}

\begin{thmUcopy}
\label{thm:Usandwiched}
\thmUcopyText
\end{thmUcopy}

We note that this result is valid for special inverse monoids in general, and that it enables one to reduce the original question to a question purely referring to the corresponding groups and their subgroups. 
In particular, our questions concerning one-relator inverse monoids are reduced to questions about subgroups of one-relator groups.
We therefore go on to prove a number of results in which we identify situations
where the conditions of the theorem are met. As a consequence, in Section \ref{sec:Examples} we exhibit several families of one-relator inverse monoids whose groups of units are all one-relator.  In particular, we  are able to prove that the group of units of the O'Hare monoid $\OHare$ is the free group of rank $2$.

It is natural to ask whether the assumptions of Theorem \ref{thm:Usandwiched} are perhaps always satisfied, or at least always in the case of one-relator presentations.
The answer to this is negative, and in order to show this we give a general construction in Section \ref{sec:Construction}
which we then apply in 
Section \ref{sec:applications} to demonstrate the following:
\begin{itemize}[leftmargin=5mm,itemsep=1mm]
\item
There exists a one-relator special inverse monoid whose group of units is not one-relator (with respect to any generating set).
\item
There exists a finitely presented special inverse monoid whose group of units is not finitely presented.
\item
There exists a one-relator special inverse monoid with finitely presented group of units, and finitely generated but non-finitely presented submonoid of right units. 
\item 
There exists a one-relator special inverse monoid whose submonoid of right units is not a free product of the groups of units and a free monoid (this follows from the previous point). 
\end{itemize}
Also, using these results, at the end of  
Section~\ref{sec:applications} we will present several results which  
show the close relationship between 
the question of 
finite presentability of the groups of units units of special one-relator inverse monoids, and an open problem of Baumslag \cite[page 76]{Baumslag73}
which asks whether every one-relator group is coherent. 
In particular we show that the problem of proving coherence of all one-relator groups is equivalent to the problem of showing that the group of units of  
an $E$-unitary one-relator special inverse monoid is always finitely presented. 

Throughout the paper, the decompositions of relators into invertible pieces play a pivotal role, and we discuss them in their own right in Section \ref{sec:Benois}. Applying a result of Narendran et al. \cite{Narendran91} we observe that, unlike the situation in one-relator monoids, for general special monoids there is no algorithm to compute the decomposition into minimal invertible pieces. The existence of such an algorithm for one-relator special inverse monoids remains open. However, we present an algorithm which does compute a decomposition into invertible pieces. The decomposition computed by this algorithm is always at least as fine as that computed by the Adjan overlap algorithm, and, in fact, in all the examples known to us in the one-relator case, including the O'Hare monoid $\OHare$, it computes the decomposition into minimal invertible pieces.

\section{Preliminaries}
\label{sec:prelim}

For an alphabet $A$, we will denote the \emph{free monoid} over $A$ by $A^\ast$.
It consists of all words over $A$, including the empty word $\epsilon=\epsilon_A$.
Since throughout words will play a somewhat ambivalent role of elements in a free monoid, as well as those in a monoid defined by a presentation, we will denote equality of words in $A^\ast$ by $u\equiv v$.

A \emph{monoid presentation} is a pair $\Mpres{A}{R}$ where $A$ is an alphabet and $R\subseteq A^\ast\times A^\ast$ is a set of pairs of words. This presentation defines the monoid $A^\ast/\rho$, where $\rho$ is the congruence
on $A^\ast$ generated by $R$. We will write $M=\Mpres{A}{R}$.
A typical relation $(u,v)\in R$ is usually written $u=v$, and we extend this to pairs in
$\rho$. Thus, in the context of the monoid defined by a presentation, $u=v$ means that $u$ and $v$,
interpreted as products of generators, represent the same element of $M$.
The monoid $M$ is a natural homomorphic image of the free monoid $A^\ast$ via the mapping $w\mapsto w/\rho$,
called the \emph{natural homomorphism}.
The \emph{group of units} $U(M)$ of a monoid $M$ consists of all elements $x\in M$ which are (two-sided) invertible,
i.e. for which there exists $y\in M$ such that $xy=yx=1_M$.
A presentation $\Mpres{A}{R}$ is called \emph{special} if every defining relation has the form $u=1$.
The word $u$ in such a relation is often called a \emph{relator}.
A \emph{one-relator monoid} is a monoid defined by a single relator, i.e. $M=\Mpres{A}{u=1}$.

For an alphabet $A$, let $A^{-1}=\{ a^{-1}\::\: a\in A\}$, be another alphabet disjoint from $A$ and in 1-1 correspondence with $A$. 
Further, let $\overline{A}=A\cup A^{-1}$, and extend the bijection $a\mapsto a^{-1}$ first to $\overline{A}$ by
$(a^{-1})^{-1}\equiv a$, and then entire $\overline{A}^\ast$ by
$(a_1^{\epsilon_1}a_2^{\epsilon_2}\dots a_n^{\epsilon_n})^{-1}
=a_n^{-\epsilon_1}\dots a_2^{-\epsilon_2}a_1^{-\epsilon_1}$
for $a_i\in A$, $\epsilon_i=\pm 1$.
The \emph{free group} $F_A$ over $A$ consists of all \emph{reduced} words from $\overline{A}^\ast$,
i.e. words containing no subword of the form $aa^{-1}$ with $a\in \overline{A}$, under multiplication
$u\cdot v=\red(uv)$. Here $\red(u)$ denotes the \emph{free reduction} of $u$, i.e. the result of successively deleting
all pairs $aa^{-1}$ in $u$. It is well known that the order of performing these deletions is inconsequential.
Sometimes we will write $\red_A$ for $\red$ when we want to work with free reductions over different alphabets.
We will also extend the use of $\red$ to apply to sets of words:
$\red(W)=\{ \red(w)\::\: w\in W\}$.

A \emph{group presentation} is a pair $\Gpres{A}{R}$, where $A$ is an alphabet, and $R\subseteq \overline{A}^\ast\times\{1\}$. The group defined by presentation
is $G=\Gpres{A}{R}=\Mpres{\overline{A}}{R,\ aa^{-1}=a^{-1}a=1\ (a\in A)}$. Alternatively, $G$ can be viewed as the quotient $F_A/N$, where
$N$ is the normal subgroup of $F_A$ generated by all $u$, where $(u,1)\in R$.
Again, $G$ is a homomorphic image of $F_A$ via the natural homomorphism $w\mapsto wN$.

The \emph{free inverse monoid} over a set $A$ will be denoted by $FI_A$. It can be viewed as the monoid
\[
\Mpres{\overline{A}}{uu^{-1}u=u,\ uu^{-1}vv^{-1}=vv^{-1}uu^{-1}\ (u,v\in\overline{A}^\ast)}.
\]

An inverse monoid presentation is a pair $\Ipres{A}{R}$ where $R\subseteq \overline{A}^\ast\times\overline{A}^\ast$.
It defines the quotient $FI_A/\rho$ where $\rho $ is the congruence generated by $R$,
which is once again a natural homomorphic image of $FI_A$ via $w\mapsto w/\rho$.
If all the relations in $R$ are of the form $u=1$, we say that $\Ipres{A}{R}$ is a \emph{special} inverse monoid.
A special inverse monoid $\Ipres{A}{u=1}$ with a single relator is called a \emph{one-relator inverse monoid}.

Every inverse monoid has a (unique) maximal group homomorphic image. For the free inverse monoid $FI_A$ this is the free group $F_A$, while for a monoid $\Ipres{A}{R}$ it is the group
$\Gpres{A}{R}$ defined by the same presentation.

For a monoid (respectively, inverse monoid or group) $M$ and a subset $X\subseteq M$, we will denote
the submonoid (resp. inverse submonoid, subgroup) that $X$ generates by $\Mgen{X}$
(resp. $\Igen{X}$, $\Ggen{X}$).

For a word $w \in A^*$ we use $\pref(w)$ to denote the set of all prefixes of $w$, and $\suff(w)$ to denote the set of suffixes.

Let $X = \{x_1, x_2, \ldots, x_n\}$ be an alphabet. We often use the notation $w(x_1, \ldots, x_n)$ to denote a word from $\overline{X}^*$ where we want to stress the fact that each letter of this word belongs to $X$ or to $X^{-1}$. Given such a word $w(x_1, \ldots, x_n)$ and given a sequence of words $p_1, \ldots, p_n$ from $\overline{Y}^*$ we use $w(p_1, \ldots, p_n)$ to denote the word from $\overline{Y}^*$ obtained by replacing each letter $x_i^\epsilon$ ($\epsilon=\pm 1$) in  the word $w(x_1, \ldots, x_n)$ by the word $p_i^\epsilon$. 

The following concept is of pivotal importance for the material in this paper:

\begin{defn}
Let $M$ be the inverse monoid defined by the presentation
\[
\Ipres{A}{r_i = 1\ (i \in I)}. 
\]
A \emph{set of invertible pieces} (or often we just say a \emph{set of pieces}) for this presentation is a collection of words $p_1, \ldots, p_k \in \overline{A}^*$ which are invertible in $M$ and satisfy $r_i \in \Mgen{p_1, \ldots, p_k} \leq \overline{A}^*$ for all $i \in \{1,2,\ldots, n\}$. 
A \emph{factorisation} of the relators $r_i$ with respect to the pieces $p_1, \ldots, p_k$ is a collection of words $r_i'$ $(i = 1, \ldots, n)$ over $k$ letters such that 
\[
r_i \equiv r_i'(p_1, \ldots, p_k) \
(i = 1, \ldots, n).
\]
For $i \in I$ by the \emph{decomposition into minimal invertible pieces} we mean the unique factorisation 
\[
r_{i}\equiv r_{i 1}r_{i 2}\ldots r_{i k_i} 
\]
with the property that each $r_{i j}$ is a non-empty word in $\overline{A}^*$ which represents an invertible element of $M$ and which has no proper non-empty invertible prefixes. 
\end{defn} 

The following facts about cancellation in an inverse semigroup will also be used throughout the paper.
They are well known and easy to prove, see \cite[Lemma 3.1, Corollary 3.2]{gray:undecidable}:

\begin{lem}
\label{lem_good}
The following hold for any inverse monoid $M$:
\begin{enumerate}[leftmargin=8mm,itemsep=1mm,label=\textup{(\roman*)}]
\item \label{it:lg1}
If $s,x\in M$ are such that $sx$ is right invertible then $sxx^{-1}=s$.
\item \label{it:lg2}
If $s,t,x\in M$ are such that $sxx^{-1}t$ is right invertible then $sxx^{-1}t=st$.
\item \label{it:lg3}
If $M=\Igen{A}$ and $w\in \overline{A}^\ast$ represents a right invertible element in $M$ then
$w=\red(w)$ in $M$.
\end{enumerate}
\end{lem}

\section{Makanin-style presentation theorems for special inverse monoids}
\label{sec:Makanin}

As discussed in the introduction, given a special inverse monoid
$M=\Ipres{A}{r_i=1\ (i\in I)}$, and decomposition of each $r_i$
into invertible minimal pieces, one may ask whether turning distinct pieces into distinct letters
yields a presentation for the group of units of $M$, as it does in the monoid case.
It is easy to see that this is not true in general. Consider for instance the one-relator inverse monoid
\[
M=\Ipres{a,b,c}{abc^2b^{-1}abc^3b^{-1}abc^2b^{-1}a=1}.
\]
We claim that the minimal invertible pieces are $a$, $bc^2b^{-1}$, $bc^{3}b^{-1}$.

To see that these three words are invertible in $M$, we can argue as follows.
First note that since $a$ is both a prefix and suffix of the defining relator it follows that $a$ is both left and right invertible, and hence is invertible. Multiplying the defining relator on the left and right by the inverse of $a$ then implies that $bc^2b^{-1}abc^3b^{-1}abc^2b^{-1}$ is invertible. Since this invertible word has $bc^2b^{-1}$ both as a prefix and as a suffix it follows that $bc^2b^{-1}$ is invertible. Finally multiplying on the left and right by the inverse of $bc^2b^{-1}$ and then the inverse of $a$ we deduce that $bc^3b^{-1}$ is invertible. A more general form or reasoning like this to obtain invertible pieces, called the Adjan overlap algorithm, will be discussed in Section~\ref{sec:Benois}.     

On the other hand $M$ is not a group, and $b$ is not invertible, because of the homomorphism
$M\rightarrow B$ onto the bicyclic monoid 
$B = \Ipres{b}{bb^{-1}=1}$ 
given by $a,c\mapsto 1$, $b\mapsto b$.
Since $b$ is not invertible but $bc^2b^{-1}$ is, it follows that $bc$ is not invertible.
Similarly, $bc^2$ and $bc^3$ are not invertible.
Therefore $a$, $bc^2b^{-1}$, $bc^{3}b^{-1}$ are indeed the minimal invertible pieces, as claimed.

Replacing these pieces by letters $x,y,z$ respectively yields the presentation
\[
G=\Gpres{x,y,z}{xyxzxyx=1} \cong F_{x,y}.
\]
In the monoid $M$ we have
\[
bcb^{-1}=(bc^3b^{-1})(bc^2b^{-1})^{-1},
\]
so the elements $a$ and $t=bcb^{-1}$ form a generating set for the group of units $U=U(M)$ of $M$.
These generators satisfy the relation
$xt^2xt^3xt^2x=1$.
Since $U$ is not free with respect to $\{x,t\}$ and $G$ is free of rank $2$ it follows that $U\not\cong G$.
In fact, $U$ is defined by the presentation $\Gpres{x,t}{xt^2xt^3xt^2x=1}$.
This will follow from the more general results we will prove below (see Corollary \ref{cor:ourMakanin}).
A key feature of this example is that $\{a,bc^2b^{-1},bc^3b^{-1}\}$ is not a basis for
$\Ggen{a,bc^2b^{-1},bc^3b^{-1}}\leq F_{a,b,c}$.
We now prove a result that gives an approach to studying groups of units of special inverse monoids which deals with this obstacle.

\begin{thm}
\label{thm:Usandwiched}
\thmUcopyText
\end{thm}

\begin{figure}

\begin{center}
{\small
\begin{tikzcd}[column sep=small]
\overline{A}^\ast \arrow[r,twoheadrightarrow,"\red",yshift =2pt] \arrow[d,twoheadrightarrow,"\pi_M"]
& 
F_A \arrow[d,twoheadrightarrow,"\pi_G"] \arrow[l,hookrightarrow,"\supseteq",yshift=-2pt] \arrow[r,hookleftarrow,"\supseteq"]
&
\begin{array}{c}
H =\Ggen{\red(r_{ij})} \\ =\Ggen{s_l} 
\end{array}
\arrow[d,twoheadrightarrow,"\pi_G\restr_H"]  
&
F_Y  \arrow[d,two heads,"\pi_K"] \arrow[l,hook',two heads,"\phi"']
\\
\begin{array}{cc}
M  \\ =\Ipres{A}{r_i=1}
\end{array}
\arrow[r,twoheadrightarrow,"\psi"] 
& 
\begin{array}{c}
G \\ =\Gpres{A}{r_i=1}
\end{array}
 \arrow[r,hookleftarrow,"\supseteq"]
&
\widetilde{H} 
&
\begin{array}{c}
K \\ =\Gpres{Y}{r_i^\prime=1}
\end{array}
 \arrow[l,two heads,"\zeta"'] \arrow[dlll,two heads,"\eta"]
\\
\begin{array}{c}
U \\ =\Igen{\pi_M(r_{ij})}\\ =\Igen{\pi_M(s_l)}
\end{array}
\arrow[u,hook,"\rotatebox{90}{$\subseteq$}"]
&&&
\end{tikzcd}
}

\caption{Monoids, groups and homomorphisms in the proof of Theorem \ref{thm:Usandwiched}.
Throughout, indexing is understood as follows: $i\in I$, $1\leq j\leq k_i$, $l\in L$.} \label{fig1}
\end{center}
\end{figure}
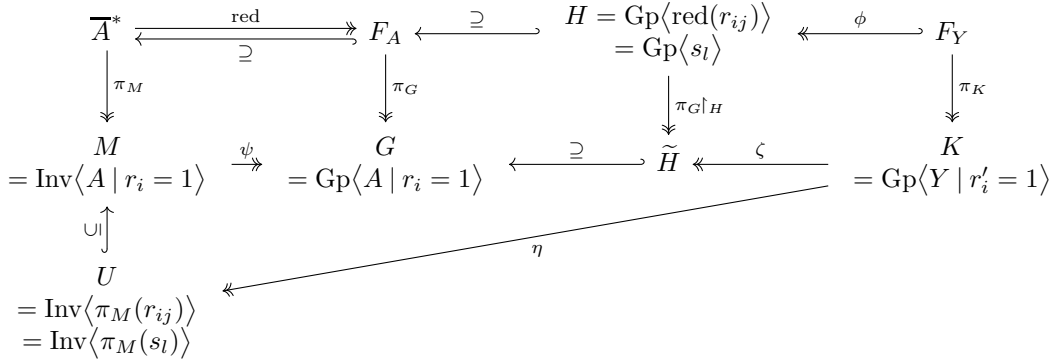

\begin{proof}
We begin working towards the diagram shown in Figure \ref{fig1}.
Start from the free monoid $\overline{A}^\ast=(A\cup A^{-1})^\ast$.
The monoid $M=\Ipres{A}{r_i=1 \; (i \in I)}$ is a natural homomorphic image of it, and we denote
by $\pi_M : \overline{A}^\ast\rightarrow M$ the natural epimorphism.

Let $F_A$ denote the free group on $A$. Recall that we regard $F_A$ as consisting of freely reduced words over $A$, and so $F_A\subseteq \overline{A}^\ast$. However, note that this inclusion is \emph{not} a homomorphism.
The group $F_A$ is a homomorphic image of $\overline{A}^\ast$, via the free reduction homomorphism
$\red:\overline{A}^\ast\rightarrow F_A$.
The group $G=\Gpres{A}{r_i=1 \; (i \in I)}$ is a natural homomorphic image of $F_A$ via
$\pi_G:F_A\rightarrow G$.
Since $G$ is also an inverse monoid, and since $M$ is \emph{defined} by $r_i=1$, $i\in I$, it follows that there exists an epimorphism 
$\psi:M\rightarrow G$ such that
\begin{equation}
\label{eq1}
\psi\pi_M=\pi_G\red.
\end{equation} 

Next, let
\[
H=\Ggen{\{ \red(r_{ij}) \::\: i\in I,\ 1\leq j\leq k_i\} }\leq F_A,
\]
and let
\[
\widetilde{H}=\pi_G(H)\leq G.
\]
As a subgroup of a free group, $H$ itself must be free, i.e. $H\cong F_Y$ for some alphabet $Y=\{ y_l\::\: l\in L\}$ as in the statement,
with an isomorphism
\begin{equation}
\label{eq5}
\phi:F_Y\rightarrow H,\ y_l\mapsto s_l\ (l\in L).
\end{equation}

Since $\red(r_{ij})\in H$, there exist $r_{ij}^\prime\in F_Y$ such that
\begin{equation}
\label{eq2}
\phi (r_{ij}^\prime)\equiv\red(r_{ij})\ (i\in I,\ 1\leq j\leq k_i).
\end{equation}
Let
\begin{equation}
\label{eq2a}
r_i^\prime \equiv \red_Y (r_{i1}^\prime r_{i2}^\prime\dots r_{ik_i}^\prime)\in F_Y\subseteq \overline{Y}^\ast,
\end{equation}
where $\red_Y$ stands for the free reduction over the alphabet $Y$ (and is not shown in the diagram).
Let
\begin{equation}
\label{eq3}
K=\Gpres{Y}{r_i^\prime =1\; (i\in I)}
=\Gpres{Y}{
\phi^{-1}(r_{i 1}) \phi^{-1}(r_{i 2})
\ldots 
\phi^{-1}(r_{i k_i}) = 1 \; (i \in I) 
},
\end{equation}
with the natural epimorphism $\pi_K : F_Y\rightarrow K$.

Next we claim that the generators $\{\pi_G(s_l)\::\: l\in L\}$ of $\widetilde{H}$ satisfy the relations
$r_i^\prime=1$, $i\in I$.
Indeed
\begin{align*}
\pi_G\phi (r_i^\prime) &= \pi_G\phi \red_Y (r_{i1}^\prime\dots r_{ik_i}^\prime) && \text{(by \eqref{eq2a})}
\\
&=\pi_G( \phi (r_{1i}^\prime)\dots \phi (r_{ik_i}^\prime)) && \text{($\phi$ is a homomorphism)}
\\
&=\pi_G(\red(r_{i1})\dots\red(r_{ik_i})) && \text{(by \eqref{eq2})}
\\
&=\pi_G\red(r_{i1}\dots r_{ik_i}) && \text{($\red$ is a homomorphism)}
\\
&=\pi_G\red(r)=1 .&&
\end{align*}
Therefore, there exists an epimorphism
\begin{equation}
\label{eq4}
\zeta:K\rightarrow \overline{H},\ \pi_K(y_l)\mapsto \pi_G(s_l)\ (l\in L).
\end{equation}

To complete the set-up, we turn our attention to the subgroup $U$ of $M$ generated by $\{ \pi_M(r_{ij})\::\: i\in I,\ 1\leq j\leq k_i\}$.

We claim that
\begin{equation}
\label{eq11}
\pi_M(s_l)\in U \quad \text{for all } l\in L.
\end{equation}
Since $s_l\in H=\Ggen{\{ \red(r_{ij})\: :\: i\in I,\ 1\leq i\leq k_i\}    }$, it follows that
\[
s_l\equiv\red(\red(r_{i_1j_1})^{\epsilon_1}\dots \red(r_{i_m j_m})^{\epsilon_m})\equiv
\red(r_{i_1 j_1}^{\epsilon_1}\dots r_{i_m j_m}^{\epsilon_m}).
\]
But all $\pi_M(r_{ij})$ are invertible in $M$ by assumption, and so
\begin{align*}
U&\ni \pi_M(r_{i_1 j_1})^{\epsilon_1}\dots \pi_M(r_{i_m j_m})^{\epsilon_m} &&\\
&=\pi_M (r_{i_1 j_1}^{\epsilon_1}\dots r_{i_m}^{\epsilon_m j_m} ) &&  \text{($\pi_M$ is  a homomorphism)}\\
&=\pi_M(\red(r_{i_1 j_1}^{\epsilon_1}\dots r_{i_m j_m}^{\epsilon_m})) && 
\text{(by Lemma \ref{lem_good}\ref{it:lg3})}\\
&=\pi_M(s_l),
\end{align*}
as required.

Next we claim that $U$ is actually generated by $\{ \pi_M(s_l)\::\: l\in L\}$.
This is proved by a very similar argument to the one in the previous paragraph, starting from the fact that 
$\{s_l\::\: l\in L\}$ is a generating set for $H$, expressing the generators $\red(r_{ij})$ in terms of the $s_l$, and using \eqref{eq11}.
Furthermore, the generators $\{ \pi_M(s_l)\::\: l\in L\}$ satisfy the relations $r_i^\prime=1$. Indeed
\begin{align*}
\pi_M\phi(r_i^\prime) &= \pi_M\phi\red_Y(r_{i1}^\prime\dots r_{ik_i}^\prime)
&& \text{(by \eqref{eq2a})}
\\
&= \pi_M(\phi(r_{i1}^\prime)\dots \phi(r_{ik_i}^\prime))
&&  \text{($\phi$ is a homomorphism)}
\\
&=\pi_M(\red(r_{i1})\dots \red(r_{ik_i})) 
&& \text{(by \eqref{eq2})}
\\
&= \pi_M\red(r_{i1})\dots \pi_M\red(r_{ik_i}) 
&&  \text{($\pi_M$ is a homomorphism)}
\\
&=\pi_M(r_{i1})\dots \pi_M(r_{ik_i})
&& \text{(by Lemma \ref{lem_good}\ref{it:lg3})}
\\
&=\pi_M(r_{i1}\dots r_{ik_i})
&& \text{($ \pi_M$ is a homomorphism)}
\\
&=\pi_M(r)=1.
\end{align*}
Therefore, there exists an epimorphism
\begin{equation}
\label{eq13}
\eta : K\rightarrow U,\ \pi_K(y_l)\mapsto \pi_M(s_l)\ (l\in L).
\end{equation}

With the foregoing set-up, the assertion of our theorem can be stated as follows:
\begin{quotation}
\textit{If $\zeta$ is an isomorphism between $K$ and $\widetilde{H}$, then $\eta$ is an isomorphism between $K$ and $U$.}
\end{quotation}
This is now actually easy to prove: it is sufficient to show that $\eta$ and 
$\zeta^{-1}\psi\restr_U$ are mutually inverse homomorphisms. To do so, in turn, it is sufficient to verify that
their two compositions act as the identity on the generating sets of $U$ and $K$ respectively.
Indeed, we have:
\begin{align*}
\eta\zeta^{-1}\psi\restr_U\pi_M(s_l) 
&= \eta\zeta^{-1}\pi_G\red (s_l) && \text{(by \eqref{eq1})}
\\
&=\eta\zeta^{-1}\pi_G(s_l) && \text{($s_l$ is reduced)}
\\
&=\eta\pi_K(y_l) && \text{(by \eqref{eq4})}
\\
&=\pi_M(s_l) && \text{(by \eqref{eq13})},
\end{align*}
and
\begin{align*}
\zeta^{-1}\psi\restr_U\eta\pi_K(y_l)
&= \zeta^{-1} \psi \pi_M(s_l) && \text{(by \eqref{eq13})}
\\
&=\zeta^{-1}\pi_G\red(s_l) && \text{(by \eqref{eq1})}
\\
&=\zeta^{-1}\pi_G(s_l) && \text{($s_l$ is reduced)}
\\
&=\pi_K(y_l) && \text{(by \eqref{eq4})},
\end{align*}
completing the proof of the theorem.
\end{proof}

It is reasonable to ask whether if one translates the presentation for $K$ given in Theorem \ref{thm:Usandwiched}
back into the alphabet $A$ the resulting presentation defines $M$.
The following theorem shows that this indeed is the case when the pieces are reduced words.

\begin{thm}
\label{thm:change-presn}
Let
$M= \Ipres{A}{r_i=1 \; (i \in I)}$
 be a special inverse monoid, let
$r_{i}\equiv r_{i 1}r_{i 2}\ldots r_{i k_i} $
be a factorisation into invertible pieces for $i\in I$, 
and suppose that all pieces are reduced words.
Let $\phi$ be an epimorphism from some free group $F_Y$ onto the subgroup of $F_A$ generated by the pieces,
and for every piece $r_{ij}$ pick a pre-image $r_{ij}^\prime\in F_Y$ under $\phi$.
Finally, let  
$\pi:\overline{Y}^* \rightarrow \overline{A}^*$ be the homomorphism extending $x \mapsto \phi(x)$ $(x \in \overline{X})$. Then 
\begin{equation}
\label{eq20}
M=\Ipres{A}{r_i=1\ (i \in I)}
=
\Ipres{A }{\pi(r_{i 1}^\prime) \ldots \pi(r_{i k_i}^\prime)=1 \ (i \in I)}
\end{equation}
and $\{ \pi(y): y \in \overline{Y} \}$ is a set of invertible pieces for $\Ipres{A }{\pi(r_{i 1}^\prime) \ldots \pi(r_{i k_i}^\prime)=1 \; (i \in I)}$ which generates the same subgroup of $M$ as the original pieces $\{ r_{i j} \::\: i \in I,\  1 \leq j \leq k_i\}$. 
\end{thm}

\begin{proof}
Denoting by $\red_A$ the free reduction over the alphabet $A$, we have
\[
\red_A\pi(w)\equiv \phi(w) \quad \text{for all } w\in \overline{Y}^\ast,
\]
and, in particular,
\begin{equation}
\label{eq21}
\red_A\pi(r_{ij}^\prime)\equiv r_{ij} \quad (i\in I,\ 1\leq j\leq k_i)
\end{equation}
because all $r_{ij}$ are reduced.

Now, we show that all the relations $\pi(r_{i 1}^\prime) \ldots \pi(r_{i k_i}^\prime)=1$ hold in the monoid $M$.
Indeed, denoting by $\pi_M:\overline{A}^\ast\rightarrow M$ the natural epimorphism, we have
\begin{align*}
\pi_M(\pi(r_{i1}^\prime)\dots \pi(r_{ik_i}^\prime))
&=
\pi_M(\red_A\pi(r_{i1}^\prime)\dots \red_A\pi(r_{ik_i}^\prime))
&& \text{(by Lemma Lemma \ref{lem_good}\ref{it:lg3})}
\\
&=
\pi_M(r_{i1}\dots r_{ik_i}) 
&& \text{(by \eqref{eq21})}
\\
&=\pi_M(r_i)=1_M.&&
\end{align*}
Furthermore, these relations imply the original ones. To see this, let $M^\prime$ be the monoid
defined by
$\Ipres{A }{\pi(r_{i 1}^\prime) \ldots \pi(r_{i k_i}^\prime)=1 \; (i \in I)}$,
and let $\pi_{M^\prime}:\overline{A}^\ast\rightarrow M^\prime$ be the associated natural epimorphism, and then
\begin{align*}
1_{M^\prime} &=
\pi_{M^\prime} (\pi(r_{i1}^\prime)\dots\pi(r_{ik_i}^\prime)) &&
\\
&=\pi_{M^\prime}(\red_A\pi(r_{i1}^\prime)\dots\red_A\pi(r_{ik_i}^\prime))
&& \text{(by Lemma \ref{lem_good}\ref{it:lg3})}
\\
&=\pi_{M^\prime}(r_{i1}\dots r_{ik_i})
&& \text{(by \eqref{eq21})}
\\
&=\pi_{M^\prime}(r_i). &&
\end{align*}
This proves \eqref{eq20}.

To prove that each $\pi(y)$ is invertible in $M$, note that by the assumptions of the theorem
$\phi(y)$ belongs to the subgroup of the free group $F_A$ generated by
$\{ r_{ij}\::\: i\in I,\ 1\leq j\leq k_i\}$,
and that $\pi(y)\equiv \phi(y)$. Hence we can write
\begin{equation}
\label{eq22}
\pi(y)\equiv \red_A(r_{i_1j_1}\dots r_{i_kj_k}).
\end{equation}
Since all $r_{ij}$ are units in $M$, by Lemma \ref{lem_good}\ref{it:lg3} we have
\begin{equation}
\label{eq23}
\pi_M(r_{i_1j_1}\dots r_{i_kj_k})=\pi_M\red_A(r_{i_1j_1}\dots r_{i_kj_k}).
\end{equation}
Combining \eqref{eq22} and \eqref{eq23} we obtain:
\[
\pi_M\pi(y) =
\pi_M \red_A(r_{i_1j_1}\dots r_{i_kj_k})=\pi_M(r_{i_1j_1}\dots r_{i_kj_k})
= \pi_M(r_{i_1j_1})\dots \pi_M(r_{i_kj_k}),
\]
which is clearly a unit.

Finally, since each relator $\pi(r_{i 1}^\prime) \ldots \pi(r_{i k_i}^\prime)$ is clearly a product of the $\pi(y)$,
it follows that $\{ \pi(y)\::\: y\in \overline{Y}\}$ is indeed a set of pieces for the presentation
$\Ipres{A }{\pi(r_{i 1}^\prime) \ldots \pi(r_{i k_i}^\prime)=1 \; (i \in I)}$.
\end{proof}

\begin{remark}
Both Theorems \ref{thm:Usandwiched} and \ref{thm:change-presn} of course apply when the pieces $r_{ij}$ are minimal. In that case, the subgroup $U$ generated by them is the group of units of $M$ by Theorem~\ref{IMM:gens}.
\end{remark}

\begin{remark}
In the statement of Theorem~\ref{thm:change-presn}, some $\pi(y)$ may actually not occur in  $\pi(r_{i 1}^\prime) \ldots \pi(r_{i k_i}^\prime)$.  
Also, we note that the assumptions in Theorem~\ref{thm:change-presn} can be weakened a little further: instead of the original pieces being reduced, one could require that $r_{ij}$ can be obtained from $\pi(r_{ij}^\prime)$ by using reductions.
\end{remark}

\begin{remark}
The statement of Theorem~\ref{thm:change-presn} is technical. 
When Theorem~\ref{thm:change-presn} is applied in the particular case where the factorizations of the defining relators are all into \emph{minimal} invertible pieces then,  
expressed in a slightly less technical way, this theorem tells us that given  
$M = \Ipres{A}{r_i=1 \; (i \in I)}$, 
then for any basis $B \subseteq F_A$ of the subgroup of $F_A$ generated by the minimal invertible pieces of the relators of this presentation, there is a presentation $\Ipres{A}{s_i=1 \; (i \in I)}$ for $M$ for which $B$ is a set of invertible pieces and also generates the group of units of $M$. 
As we shall see below, changing the presentation for $M$ in this way will often be a useful first step when analysing examples.   
\end{remark}

We note that the assumption that the pieces are reduced is necessary in Theorem \ref{thm:change-presn}.
Indeed, in the monoid $\Ipres{x,y}{xx^{-1}yy^{-1}xx^{-1}=1}$, the minimal invertible pieces are $xx^{-1}$ and $yy^{-1}$.
In $F_{x,y}$ they generate the trivial subgroup. However, the original monoid is \emph{not} isomorphic
to the free inverse monoid $\Ipres{x,y}{}$.

Motivated by Theorem \ref{thm:Usandwiched} we introduce the following concept:

\begin{defn}
\label{defn:freefor}
A finite set of words $w_1, \ldots, w_k \in \overline{A}^*$ is \emph{free for substitution
in a group presentation}
\[
\Gpres{x_1,\dots,x_k}{r_i(x_1, \ldots, x_k) = 1\ (i\in I)}
\]
 if  the subgroup of 
\[
\Gpres{A}{r_i(w_1, \ldots, w_k)=1\ (i\in I)}
\] 
generated by $w_1, \ldots, w_k$ is isomorphic to $\Gpres{x_1,\dots,x_k}{r_i(x_1, \ldots, x_k) = 1\ (i\in I)}$ via the map $x_i \mapsto w_i$. 
\end{defn}

\begin{cor}
\label{cor:freefor}
Let $A$ be an alphabet, and suppose that $r_1,\dots, r_k\in\overline{A}^\ast$ are such that the free group
$H=\Ggen{r_1,\dots,r_k}\leq F_A$ has a basis which is free for substitution into any one-relator presentation.
Then the group $U$ of units of any
one-relator inverse monoid $M=\Ipres{A}{r=1}$, where the minimal invertible pieces of $r$ are precisely
$ r_1,\dots, r_k$, is again a one-relator group.
Specifically, if $\{ p_1,\dots, p_n\}$ is a basis for $H$, if each $r_j$ is expressed in terms of this basis
as $r_j=r_j^\prime(p_1,\dots,p_n)$, where $r_j^\prime\in \overline{\{x_1,\dots,x_n\}}^\ast$, and if
$r=r^\prime(r_1,\dots, r_k)$ then $U$ is defined by
\[
U=\Gpres{x_1,\dots,x_n}{r^\prime(r_1^\prime,\dots,r_k^\prime)=1}.
\]
\end{cor}

\begin{proof}
This follows immediately from Definition \ref{defn:freefor} and Theorem \ref{thm:Usandwiched}.
\end{proof}

It is therefore of interest to investigate sets that are free for substitutions into one-relator presentations, or indeed, into arbitrary presentations.
For instance, we can prove the following:

\begin{thm} \label{thm:markers}
Let $A$ be an alphabet, let  $p_1, \ldots, p_n \in F_A$, and denote by $A_i$ the set of letters of $A$ that occur in $p_i$. 
If 
\[
A_i \not\subseteq 
\bigcup_{j \neq i}{A_j} \quad
\text{for all } i = 1, \ldots, n
\]
 then $\{p_1,\dots, p_n\}$ is free for substitutions into any one-relator presentation.
\end{thm}

The following two lemmas will be used in the proof:

\begin{lem}
\label{lem:2rel}
If $G = \Gpres{A}{r_1=1, r_2=1}$ where each $r_i$ is cyclically reduced and contains a letter $a_i$ not contained by the other $r_j$, then the one-relator groups $G_1=\Gpres{A \setminus \{a_2\}}{r_1=1}$ and $G_2=\Gpres{A \setminus \{a_1\}}{r_2=1}$ both embed naturally into $G$.  
\end{lem}

\begin{proof}
The subgroup of $G_1$ generated by $A\setminus\{ a_1,a_2\}$ is free on that set by Freiheitsatz, since the letter $a_1$ occurs in $r_1$.
Likewise, writing $G_2=\Gpres{A^\prime \setminus \{a_1^\prime\}}{r_2^\prime=1}$
over a disjoint copy of $A$, the
subgroup generated by $A^\prime\setminus\{a_1^\prime,a_2^\prime\}$ is free.
So we can form the free product with amalgamation
\[
\Gpres{(A\setminus\{ a_2\})\cup (A^\prime\setminus\{ a_1^\prime\})}{
r_1=1,\ r_2^\prime=1,\ a=a^\prime\ (a\in A\setminus\{a_1,a_2\})}
\]
into which $G_1$ and $G_2$ will embed naturally.
Now perform Tietze transformations to remove the generators $a^\prime$ with
$a\in A\setminus\{a_1,a_2\}$, rename $a_2^\prime $ into $a_2$, 
and we obtain $\Gpres{A}{r_1=1,\ r_2=1}=G$.
\end{proof}

\begin{lem}
\label{lem:lettersappear}
Let $A$ and $X = \{ x_1, \ldots, x_n \}$ be two alphabets, and 
let $r = r(x_1, \ldots, x_n) \in F_X$ be a cyclically reduced word in which all of the letters 
$x_1, \ldots, x_n$ appear.
Further, let $p_1, \ldots, p_n \in F_A$, and denote by $A_i$ the set of letters of $A$ that appear in $p_i$. Suppose that  
\[
A_i \not\subseteq 
\bigcup_{j \neq i}{A_j} 
\quad
\text{for all } i = 1, \ldots, n,
\]
with $a_i\in A_i \setminus \bigcup_{j \neq i}{A_j}$. 
Then each of the letters $a_i$ $(i = 1, \ldots, n)$ appears in any cyclic reduction
of the word
$r(p_1, \ldots, p_n)$.
\end{lem}

\begin{proof}
Since $r$ is cyclically reduced, and since we are concerned with cyclically reduced forms
of $r(p_1,\dots,p_n)$, we may without loss of generality assume that $r$ has the form
\[
r\equiv x_{i_1}^{\alpha_1} x_{i_2}^{\alpha_2}\dots x_{i_k}^{\alpha_k}
\]
where $ i_j\in\{1,\dots,n\}$ and $ \alpha_j\in\mathbb{Z}$ for $ j=1,\dots, k$, and
\[
i_1\neq i_2\neq i_3\dots \neq i_k\neq i_1.
\]
Consider the word
\[
r(p_1,\dots, p_n)\equiv p_{i_1}^{\alpha_1} p_{i_2}^{\alpha_2}\dots p_{i_k}^{\alpha_k},
\]
and inside it, for every $j=1,\dots, k$, consider all the occurrences of the letter 
$a_{i_j}$ 
appearing in the reduced word $\red(p_{i_j}^{\alpha_j})$.
Notice that there must be at least one such occurrence, since $p_{i_j}$ is reduced,
and hence the set of letters appearing in $p_{i_j}$ is the same as the set of
letters appearing in  $\red(p_{i_j}^{\alpha_j})$.
We claim that none of these occurrences are cancelled in the process of free reduction
of $r(p_1,\dots, p_n)$.

To see this, suppose to the contrary that two occurrences of some $a_i$ cancel each other.
Let those two occurrences appear in $p_{i_l}^{\alpha_l}$ and $p_{i_m}^{\alpha_m}$, where
$i=i_l=i_m$ and $l\leq m$. Furthermore, choose $i$, $l$ and $m$ so that $m-l$ is as small as possible.
By assumption, the occurrence of $a_i$ in $p_{i_l}$ is not cancelled when reducing to
$\red(p_{i_l}^{\alpha_l})$, and
hence we cannot have $m=l$.
Also, since $i_{l+1}\neq i_l$, the word $p_{i_{l+1}}$ contains no occurrences of $a_i=a_{i_l}$, and so
$m\neq l+1$. Therefore we must have $m\geq l+2$.
But now, in order for the two occurrences of $a_i$ in $p_{i_l}^{\alpha_l}$ and $p_{i_m}^{\alpha_m}$
to cancel each other, any occurrence of $a_{i_{l+1}}$ in
$p_{i_{l+1}}^{\alpha_{l+1}}$ must also be cancelled. Since there are no occurrences of $a_{i_{l+1}}$
in $p_{i_l}^{\alpha_l}$, it follows that these occurrences must cancel within
$p_{i_{l+1}}^{\alpha_{l+1}}\dots p_{i_m}^{\alpha_m}$, and this contradicts the minimality of $m-l$.
This proves the claim that no two occurrences of some $a_i$ within 
some $p_{i_l}^{\alpha_l}$ and $p_{i_m}^{\alpha_m}$ ($i_l=i_m=i$)
cancel each other
in the process of free reduction of $r(p_1,\dots,p_n)$.

Now suppose that a cancellation of some $a_{i_l}$ appearing in some $p_{i_l}^{\alpha_l}$
can happen during the cyclic reduction of $r(p_1,\dots,p_n)$.
Then this occurrence of $a_{i_l}$ would be cancelled during the free reduction of at least one of the words
\[
p_{i_l}^{\alpha_l} \dots p_{i_{k}}^{\alpha_{k}} p_{i_{1}}^{\alpha_{1}}\dots p_{i_{l-1}}^{\alpha_{l-1}}
\quad\text{or}\quad
p_{i_{l+1}}^{\alpha_{l+1}}\dots p_{i_{k}}^{\alpha_{k}} p_{i_{1}}^{\alpha_{1}}\dots p_{i_l}^{\alpha_l}.
\]
However, the words
\[
x_{i_l}^{\alpha_l} \dots x_{i_{k}}^{\alpha_{k}} x_{i_{1}}^{\alpha_{1}}\dots x_{i_{l-1}}^{\alpha_{l-1}}
\quad\text{and}\quad
x_{i_{l+1}}^{\alpha_{l+1}}\dots x_{i_{k}}^{\alpha_{k}} x_{i_{1}}^{\alpha_{1}}\dots x_{i_l}^{\alpha_l}
\]
satisfy the original assumptions made about $r$, and hence such cancellation cannot take place by the argument from the previous paragraph.
This completes the proof of the lemma.
\end{proof}

\begin{proof}[Proof of Theorem \ref{thm:markers}]
Let $X=\{ x_1,\dots, x_n\}$. Suppose $p_1,\dots, p_n\in\overline{A}^\ast$,
and let $A_i$ be the set of letters appearing in $p_i$. For every $i=1,\dots, n$ pick
$a_i\in A_i\setminus\bigcup_{j\neq i} A_j$.
We need to show that for every $r\in \overline{X}^\ast$ the subgroup
of
\begin{equation}
\label{eq31}
G=\Gpres{A}{r(p_1,\dots, p_n)=1}
\end{equation}
generated by $p_1,\dots, p_n$ is naturally isomorphic to 
$\Gpres{X}{r=1}$.
Clearly, we may assume without loss of generality that $r$ is cyclically reduced.

We will prove the assertion by induction on the length of the word $r(p_1,\dots,p_n)$.
If $|r(p_1,\dots,p_n)|=0$, Lemma \ref{lem:lettersappear} gives $|r|=0$, and the assertion is obvious.
It is also obvious when $|p_i|=1$ for all $i=1,\dots,n$.
So, suppose that some $p_i$ has length greater than $1$.
Without loss we may assume $|p_n|>1$.
In the presentation \eqref{eq31} for $G$ introduce a new generator $x_n$ satisfying $x_n=p_n$:
\begin{align}
\label{eq33a}
G&=\Gpres{A,x_n}{p_nx_n^{-1}=1,\ r(p_1,\dots,p_{n-1},x_n)=1}
\\
\label{eq33}
&=\Gpres{A,x_n}{p_nx_n^{-1}=1,\ \overline{r}=1},
\end{align}
where $\overline{r}$ denotes a cyclically reduced form of $r(p_1,\dots,p_{n-1},x_n)$.
We claim that the presentation \eqref{eq33} satisfies the assumptions of Lemma \ref{lem:2rel}.
Indeed:
\begin{itemize}[leftmargin=5mm,itemsep=1mm]
\item
The word $p_nx_n^{-1}$ is cyclically reduced because $p_n$ is a reduced word over $\overline{A}$ and $x_n\not\in\overline{A}$.
\item
The word $\overline{r}$ is cyclically reduced by assumption.
\item
The letter $a_n$ appears in $p_n$ and in none of $p_1,\dots,p_{n-1}$;
hence $a_n$ will appear in $p_nx_n^{-1}$ but not in $r(p_1,\dots,p_{n-1},x_n)$,
and hence not in $\overline{r}$ either.
\item
The letter $a_1$ appears in $p_1$, but not in any other $p_2,\dots,p_n$;
hence it will appear in $r(p_1,\dots,p_{n-1},x_n)$ by Lemma \ref{lem:lettersappear},
but not in $p_nx_n^{-1}$.
\end{itemize}
By Lemma \ref{lem:2rel}, the group
\begin{equation}
\label{eq34}
\Gpres{A\setminus\{a_n\},x_n}{\overline{r}=1}=\Gpres{A\setminus\{a_n\},x_n}{
r(p_1,\dots,p_{n-1},x_n)}
\end{equation}
embeds naturally into $G$ as defined by \eqref{eq33}.
Note that $|r(p_1,\dots,p_{n-1},x_n)|<|r(p_1,\dots,p_n)|$ because
$|p_n|\geq 2$.
Furthermore, note that the words $p_1,\dots,p_{n-1},x_n$ satisfy all the original assumptions about $p_1,\dots, p_n$ -- they are reduced, and each contains a letter not appearing in any of the others.
Therefore, the inductive hypothesis applies, and the subgroup of the group \eqref{eq34} generated
by $p_1,\dots, p_{n-1},x_n$ is naturally isomorphic to $\Gpres{X}{r=1}$.
But, inside $G$, this subgroup coincides with the subgroup generated by $p_1,\dots,p_n$,
and the theorem is proved.
\end{proof}

We can generalise the above condition a bit further:

\begin{thm}
\label{thm:posets}
Let  $A=B\cup C$ where $B$ and $C$ are disjoint alphabets,
and let $p_1,\dots,p_n\subseteq  F_A$. 
Denoting by $c(p_i)$ the set of letters from $C$ that appear in $p_i$, 
suppose that there is a bijection $\mu : \{p_1,\dots,p_n\}\rightarrow C$ such that at least one of the following two conditions holds for every $i=1,\dots, n$:
\begin{enumerate}[leftmargin=8mm,itemsep=1mm,label=\textup{(\roman*)}]
\item \label{it:pos1}
$c(p_i) = \{ \mu(p_i) \}$; or 
\item \label{it:pos2}
$\mu(p_i)$ occurs precisely once in $p_i$ and 
\[
c(p_i) \setminus \bigcup_{1 \leq j \leq i-1}c(p_j) = \{ \mu(p_i) \}.
\]
\end{enumerate}
Then $p_1,\dots, p_m$ are free for substitutions into arbitrary one-relator presentations. 
\end{thm} 

\begin{remark}
Theorem \ref{thm:markers} is indeed a special case of Theorem \ref{thm:posets},
corresponding to the case where condition \ref{it:pos1} is satisfied for every $i=1,\dots,n$.
\end{remark}

\begin{proof}
We may reorder $p_1,\dots, p_n$ so that $p_1,\dots, p_m$ satisfy \ref{it:pos1} and
$p_{m+1},\dots p_n$ satisfy \ref{it:pos2} for some $1\leq m\leq n$.
For $i=1,\dots ,n$ let $c_i=\mu(p_i)$ so that $C=\{c_1,\dots,c_n\}$, and let $x_1,\dots, x_n,d_{m+1},\dots, d_n$
be new letters, not in $A$.

We need to prove that for every word $r=r(x_1,\dots,x_n)$ over $\overline{\{x_1,\dots,x_n\}}$
the subgroup of
\begin{equation}
\label{eq:pos2}
\Gpres{A}{r(p_1,\dots,p_n)=1}
\end{equation}
generated by $p_1,\dots,p_n$ is naturally isomorphic to $\Gpres{X}{r=1}$, where $X = \{x_1, \ldots, x_n \}$.
To begin with, by Theorem \ref{thm:markers}, the subgroup of
\[
\Gpres{B,c_1,\dots,c_m,d_{m+1},\dots, d_n}{r(p_1,\dots,p_m,d_{m+1},\dots, d_n)=1}
\]
generated by $p_1,\dots,p_m,d_{m+1},\dots, d_n$ is naturally isomorphic to
$\Gpres{X}{r=1}$. We treat this as the inductive anchor for 
the claim that for every $k=0,\dots, n-m$, the subgroup of
\begin{equation}
\label{eq:pos1}
\Gpres{B,c_1,\dots,c_{m+k},d_{m+k+1},\dots, d_n}{
r(p_1,\dots, p_{m+k},d_{m+k+1},\dots, d_n)=1}
\end{equation}
generated by $p_1,\dots,p_{m+k},d_{m+k+1},\dots, d_n$ is naturally isomorphic to
$\Gpres{X}{r=1}$.
Suppose the claim is true for some $k$, and we prove it for $k+1$.
First we write $p_{m+k+1}\equiv qc_{m+k+1}r$, where
$q,r\in \overline{B\cup\{c_1,\dots, c_{m+k}\}}^\ast$, because of \ref{it:pos2}.
Then we introduce a new generator $c_{m+k+1}$ into \eqref{eq:pos1}, via the relation
$c_{m+k+1}=q^{-1}d_{m+k+1} r^{-1}$. Then we use this relation to eliminate 
the generator $d_{m+k+1}$ via $d_{m+k+1}=qc_{m+k+1}r=p_{m+k+1}$.
Thus we obtain the presentation
\[
\Gpres{B,c_1,\dots,c_{m+k},c_{m+k+1},d_{m+k+2},\dots, d_n}{
r(p_1,\dots, p_{m+k},p_{m+k+1},d_{m+k+2},\dots, d_n)=1}
\]
which defines a group naturally isomorphic to the group defined by \eqref{eq:pos1}.
The subgroup of this group generated by $p_1,\dots,p_{m+k+1},d_{m+k+2},\dots,d_n$
is isomorphic to the subgroup of \eqref{eq:pos1} generated by $p_1,\dots,p_{m+k},d_{m+k+1},\dots,d_n$, which, in turn, is isomorphic to $\Gpres{X}{r=1}$ by induction.
This completes the inductive proof. Putting $k=n-m$ we obtain that the subgroup of
\eqref{eq:pos2} generated by $p_1,\dots,p_n$ is naturally isomorphic to
$\Gpres{X}{r=1}$, and this completes the proof of the theorem.
\end{proof}

\begin{cor}
\label{cor:ourMakanin}
Let $M = \Ipres{A}{r=1}$ and let $r \equiv r_1 \ldots r_k$ be the factorisation of $r$ into minimal invertible pieces. If there is a basis $p_1, \ldots, p_n$ of $\Ggen{r_1, \ldots, r_k} \leq F_A$ satisfying the condition in the 
statement of Theorem~\ref{thm:posets}  then the group of units of $M$ is a one-relator group. 
Specifically, expressing each $r_j=r_j^\prime(p_1,\dots, p_n)$ in $F_A$, 
where $r_j^\prime\in \overline{\{x_1,\dots,x_n\}}^\ast$, the group of
units of $M$ is defined by
\[
\Gpres{x_1,\dots,x_n}{r_1^\prime \dots r_k^\prime=1}.
\]
\end{cor}

\begin{proof}
This follows immediately from Theorem \ref{thm:posets} and Corollary \ref{cor:freefor}. 
\end{proof}

As a special case of Theorem \ref{thm:posets} we see that any sets $\{p_1,\dots,p_n\}\subseteq F_A$ of the following types are free for substitutions into any one-relator presentations:
\begin{enumerate}[leftmargin=10mm,itemsep=1mm,label=\textup{(F\arabic*)}]
\item
\label{it:F1}
$p_i$s are powers of distinct generators from $A$
(this is Corollary 4.10.2 in \cite{mks});
\item
\label{it:F2}
$p_i$s are words over disjoint subalphabets of $A$;
\item
\label{it:F3}
for each $p_i$ there exists a letter $a_i\in A$ which appears in $p_i$ precisely once, and does not appear in any $p_j$ with $j<i$.
\end{enumerate}
We do not know whether the assumptions in  Theorems \ref{thm:posets} or \ref{thm:markers} actually imply that the sets in question are free for substitutions in \emph{any} presentations. However, this is the case for the above-listed special instances as we shall see in the next theorem.

At this point it is very natural to wonder whether for a set of pieces to be free for substitutions it might suffice just to assume that the set of words is Nielsen reduced. 
We will see later in Subsection~\ref{subsection_Higman} that the answer to this question is no; see Proposition~\ref{thm_Higman}. 
\begin{thm}
\label{thm:FreeInAny}
If $p_1,\dots, p_n\in F_A$ satisfy any of 
\ref{it:F1}, \ref{it:F2} or \ref{it:F3}
then $\{p_1,\dots, p_n\}$ is free for substitution into any presentation.
\end{thm}

\begin{proof}
\ref{it:F1} is a special case of \ref{it:F2}, so there is no need to prove it separately. Suppose that \ref{it:F2} is satisfied.
We need to prove that for any group
$G=\Gpres{X}{r_i=1\ (i\in I)}$, with $X=\{x_1,\dots,x_n\}$, it naturally embeds
into
\begin{equation}
\label{eq:FIA1}
\Gpres{A}{r_i(p_1,\dots,p_n)=1\ (i\in I)}.
\end{equation}
To this end, for $k=1,\dots,n$, let $A_k\subseteq A$ be the set of letters that appear in $p_i$,
let $A^\prime=A\setminus(A_1\cup\dots\cup A_n)$, and let
\[
G_k=\Gpres{A_1,\dots,A_{k-1},A^\prime,x_k,\dots, x_n}{
r_i( p_1,\dots,p_{k-1},x_k,\dots,x_n)=1\ (i\in I  )}.
\]
We claim that for all $k=1,\dots,n-1$, the group $G_k$ embeds naturally into $G_{k+1}$ via
\begin{alignat*}{2}
a &\mapsto a &\hspace{10mm}& (a\in A_1\cup\dots\cup A_{k-1}\cup A^\prime),\\
x_k&\mapsto p_k, &&\\
x_j&\mapsto x_j &&(j=k+1,\dots,n).
\end{alignat*}
Once the claim is proved, composing all these embeddings together with the obvious embedding
of $\Gpres{X}{r_i=1\ (i\in I)}$ into
\[
G_1=\Gpres{A^\prime,x_1,\dots,x_n}{r_i(x_1,\dots,x_n)=1\ (i\in I)}
=\Gpres{X}{r_i=1\ (i\in I)}\ast F_{A^\prime}
\]
gives the desired embedding of $\Gpres{X}{r_i=1\ (i\in I)}$ into
\eqref{eq:FIA1}.

To prove the claim, let $m$ be the order of $x_k$ in $G_k$ if this is finite, and otherwise let $m=0$.
Note that by \cite[Corollary 4.4.11]{mks}, the element $p_k$ in the group $\Gpres{A_k}{p_k^m=1}$ has finite order if and only if $m>0$, and that in this case the order is precisely $m$.
It follows that we can form the free product with amalgamation
of the group $G_k$ with the group  $\Gpres{A_k}{p_k^m=1}$ amalgamating
the subgroup $\Ggen{x_k}$ of the former with $\Ggen{p_k}$ of the latter.
This group naturally embeds $G_k$, and has the presentation
\[
\Gpres{A_1,\dots,A_k,A^\prime,x_k,\dots,x_n}{
r_i(p_1,\dots,p_{k-1},x_k,\dots,x_n)=1\ (i\in I),\ 
p_k^m=1,\ x_k=p_k}.
\]
Eliminate $x_k$ using the last relation in this presentation, to obtain
\[
\Gpres{A_1,\dots,A_k,A^\prime,x_{k+1},\dots,x_n}{
r_i(p_1,\dots,p_{k-1},p_k,\dots,x_n)=1\ (i\in I),\ 
p_k^m=1}.
\]
Notice that the specified mapping is certainly a homomorphism, and hence
$p_k^m$ is a consequence of the 
relations $r_i(p_1,\dots,p_k,x_{k+1},\dots,x_n)=1$ ($i\in I$).
Eliminating it yields the presentation for $G_{k+1}$ and proves the claim,
and hence this case of the theorem.

Suppose now that (F3) holds.
Write
\[
p_j\equiv p_j^\prime a_j^{\epsilon_j}p_j^{\prime\prime},\quad
j=1,\dots,n,
\]
and let $A^\prime=A\setminus\{a_1,\dots,a_n\}$.
Let $X=\{x_1,\dots,x_n\}$ and consider arbitrary
$r_i=r_i(x_1,\dots,x_n)\in\overline{X}^\ast$ ($i\in I$).
We need to show that the group
\begin{equation}
\label{eq:FIA12}
\Gpres{X}{r_i(x_1,\dots,x_n)=1\ (i\in I)}
\end{equation}
naturally embeds into
\begin{equation}
\label{eq:FIA13}
\Gpres{A}{r_i(p_1,\dots,p_n)=1\ (i\in I)}.
\end{equation}
It is clear that \eqref{eq:FIA12} embeds into
\begin{equation}
\label{eq:FIA11}
\Gpres{A^\prime,X}{r_i(x_1,\dots,x_n)=1\ (i\in I)}=
\Gpres{X}{r_i(x_1,\dots,x_n)=1\ (i\in I)}\ast F_{A^\prime}.
\end{equation}
By assumption $p_j^\prime,p_j^{\prime\prime}\in \overline{A^\prime}^\ast$, so 
in \eqref{eq:FIA11} we can introduce redundant generators
$a_1,\dots,a_n$ as follows:
\begin{align*}
&\Gpres{A^\prime,a_1,\dots,a_n,X}{r_i(x_1,\dots,x_n)=1\ (i\in I),\ 
a_j^{\epsilon_j}=(p_j^\prime)^{-1}x_j(p_j^{\prime\prime})^{-1}\ (j=1,\dots,n)}
\\
=&
\Gpres{A,X}{r_i(x_1,\dots,x_n)\ (i\in I),\ 
x_j=p_j\ (j=1,\dots,n)}.
\end{align*}
Eliminating the $x_j$ from the last presentation yields the group \eqref{eq:FIA13}, showing that
it is isomorphic to \eqref{eq:FIA11}, and thus naturally embedding \eqref{eq:FIA12}, as required.
\end{proof}

\begin{cor}
\label{cor:ourMakanin1}
Let $M = \Ipres{A}{r_i=1\ (i\in I)}$ be any special inverse monoid, and let $r_i \equiv r_{i1} \ldots r_{ik_i}$ be the factorisation of $r_i$ into minimal invertible pieces. If there is a basis $p_1, \ldots, p_n$ of $\Ggen{\{ r_{ij}\::\: i\in I,\ 1\leq j\leq k_i\}} \leq F_A$ satisfying any of 
\ref{it:F1}, \ref{it:F2} or \ref{it:F3},  then the group of units $U$ of $M$ can be defined by a presentation with generators $A$ and $|I|$ many defining relations.
Specifically, if each $r_{ij}$ is written as $r_{ij}=r_{ij}^\prime(p_1,\dots,p_n)$, where
$r_{ij}^\prime\in\overline{\{x_1,\dots,x_n\}}^\ast$ then $U$ is defined by
\[
U=\Gpres{x_1,\dots,x_n}{r_{i1}^\prime\dots r_{ik_i}^\prime=1\ (i\in I)}.
\]
\end{cor}

\begin{proof}
This analogous to Corollary \ref{cor:ourMakanin}.
\end{proof}

\section{An approach to computing pieces via regular languages}
\label{sec:Benois}

There already is in the literature an algorithm for computing a decomposition into invertible pieces. It was originally introduced by Adjan, and modified by Zhang. It 
was designed for the setting of special monoid presentations, rather than inverse monoid presentations, but in fact remains valid for the latter as well. 

The description we now give of the Adjan overlap algorithm follows 
that given by Lallement in \cite{Lallement1974}, who was chiefly concerned with
 one relator monoids $\Mpres{A}{r=1}$.  
Here we shall explain the natural generalisation of that algorithm to arbitrary finitely presented special monoids 
$\Mpres{A}{r_i \; (i \in I)}$. 
In general, this algorithm will compute a decomposition of the relators into pieces, but  
these will not in general be the minimal pieces; see Example~\ref{ex:AdjanFailsForNonOR} below. 
However, in the particular case of one-relator monoids $\Mpres{A}{r=1}$ an important theorem of Adjan \cite{adjan} shows that this algorithm does compute the minimal invertible pieces of the defining relator $r$.

Let $M= \Mpres{A}{r_i=1 \; (i \in I)}$ be a special monoid and set $R = \{r_i \::\: i \in I\}$. 
We say that a submonoid $T$ of the free monoid $A^*$ has \emph{property $(\A)$} if    

\begin{enumerate}[leftmargin=10mm,itemsep=1mm,label=\textup{($\A$)}]
\item For all $\alpha, \beta, \gamma \in A^*$, if $\alpha\beta \in T$ and $\beta\gamma \in T$ then $\{\alpha,\beta,\gamma\} \subseteq T$.      
\end{enumerate}

Let $U(R)$ be the smallest submonoid of $A^*$ such that $R \subseteq U(R)$ and $U(R)$ satisfies ($\mathcal{A}$).

To see that the monoid $U(R)$ exists note that the set of submonoids of $A^*$ containing $R$ and satisfying ($\A$) is non-empty (since e.g. $A^*$ satisfies these properties), and it follows from the definitions that the intersection of any family of submonoids of $A^*$ with these properties again gives a submonoid of $A^*$ with these properties.  

Recall that a subset $C$ of $A^+$ is called a \emph{code} if $C$ is a set of free generators for the submonoid $\Mgen{C} \leq A^+$.     A subset $C$ of $A^+$ is said to have the \emph{prefix property} if $CA^+ \cap C = \varnothing$.  Similarly, we say $C$ has the \emph{suffix property} if $C \cap A^+C = \varnothing$.   Any subset $C$ of $A^+$ with the prefix property is a code. We call these \emph{prefix codes}.  Similarly, if $C$ has the suffix property, then $C$ is a code called a \emph{suffix code}.  If $C$ is both a prefix and a suffix code, then we call $C$ a \emph{biprefix code}.

We now make some observations about the set $U(R)$.    
\begin{enumerate}
\item[(a)] Every word $w \in U(R)$ represents and invertible element of the monoid $M$.   
\item[(b)] $U(R)$ is a free monoid which is freely generated by a finite biprefix code $B(R)$. 
\item[(c)] There is an algorithm which takes any finite set $R$ of words as input and computes the biprefix code $B(R)$.  
\end{enumerate}
For proofs of these facts we refer the reader to \cite[Section~1]{Lallement1974}.  
The idea behind the proof of (a) is to observe that if $\gamma$ and $\delta$ are words that represent invertible elements of $M$ then so does their product, and also if $\alpha \beta$ and $\beta \gamma$ are words that represent invertible elements of $M$, then the words $\alpha$, $\beta$ and $\gamma$ all also represent invertible elements of $M$.     
The starting point for the algorithm is provided by the members of $R$, which, being equal to $1$ in $M$, certainly represent invertible elements.
Since $U(R)$ is obtained from $R$ by closing under taking products and under adding triples of words $\alpha, \beta$ and $\gamma$ to ensure condition ($\A$), this can be used to prove (a). 

Regarding (a), it is important to note that even for one relator monoids we are not claiming that every word which represents an invertible element of $M$ necessarily belongs to the set $U(R)$. Indeed if one considers the simple example of the bicyclic monoid $\Mpres{b,c}{bc=1}$ then clearly $U(R) = \{bc\}^*$ is the smallest submonoid of $\{b,c\}^*$ which contains $\{bc\}$ and has property ($\A$). But the word $bbcc$ is equal to $1$ in the bicyclic monoid, so certainly represents an invertible element, while $bbcc \not\in \{bc\}^* = U(R)$.

Even though we do not need it in this paper, for the sake of completeness we explain here an algorithm which computes the finite biprefix code $B(R)$ that generates the free monoid $U(R)$. 
This description follows \cite[p372]{Lallement1974} (with a small correction we have made to a typo in (i) on \cite[p372]{Lallement1974}).  

For $i \in \mathbb{N}$ we define the set $W_i$ inductively as follows. 
First set $W_0 = R = \{r_i \::\: i \in I \}$.  
Then for $i>0$ we define $W_i$ inductively in the following way.   
For $u \in A^*$ we have $u \in W_{i+1}$ if and only if one of the following conditions holds:
\begin{enumerate}
\item[(i)] $u \in W_i$ and $v \not\in \pref(u)$ and $v \not\in \suff(u)$ for all $v \in W_i \setminus \{ u \}$; 
\item[(ii)] there exist $v, v' \in X^*$ not both equal to $\epsilon$ such that $uv \in W_i$ and $v'u \in W_i$;
\item[(iii)] there exist $v, v' \in X^*$ with $v \neq \epsilon$ such that $uv \in W_i$ and $vv' \in W_i$; 
\item[(iv)] there exist $v,v' \in X$ with $v' \neq \epsilon$ such that $v'u \in W_i$ and $vv' \in W_i$. 
\end{enumerate} 
It may be shown that there is a number $n$ such that $W_n = W_{n'}$ for all $n' \geq n$ i.e. this process eventually stabilises. 
Then $B(R) = W_n$ is a biprefix code that freely generates $B(R)$; see \cite[Theorem~1(i)]{Lallement1974}.   
Roughly speaking, the idea behind the above algorithm is that steps (ii)-(iv) are designed to ensure that the set we construct does satisfy property ($\A$), while step (i) ensures that the set constructed has the prefix and suffix properties. 
Since $B(R)$ is a code which generates the free monoid $U(R)$ it follows that every word $w \in U(R)$ can be written uniquely as $w \equiv w_1 \ldots w_k$ where $w_i \in B(R)$ for $1 \leq i \leq k$. 
Since $R \subseteq U(R)$ it follows that in particular every relator word $r_i$ with $i \in I$ decomposes uniquely as a product of words from $B(R)$.  
We call this the decomposition of relators into invertible pieces determined by 
the Adjan algorithm.
The key result showing the importance of this algorithm is that for special one-relator monoids this algorithm actually computes the decomposition of the relator into minimal invertible pieces. 
We have already seen that the presentation determined by these pieces is then a presentation for the group of units of the monoid (see Theorem~\ref{thm:SpecialMakanin}).  
This gives the following result of Adjan \cite{adjan}. 

\begin{thm}
[Adjan {\cite[Chapter III, Theorem 7]{adjan}\label{thm:adjan:main}}]
For any one-relator special monoid $\Mpres{A}{r=1}$ the Adjan algorithm computes the decomposition of $r$ into minimal invertible pieces. 
More precisely, the biprefix code $B(\{r\})$ computed by the Adjan algorithm 
is equal to the set of minimal invertible pieces of the relator $r$.   
Hence there is an algorithm that takes any special one-relator monoid as input, and computes a one-relator presentation which is isomorphic to the group of units of the monoid. 
  \end{thm}

For arbitrary special monoid presentations this result is no longer true, as the following example demonstrates. 

\begin{example}\label{ex:AdjanFailsForNonOR}
Consider the monoid presentation 
\[
\Mpres{a,b,c,d}{
ab=1, cabd=1, cdd=1
}
\]
There are no overlaps, but $d = 1$ is easily seen to be a consequence of the defining relations. In particular the letter $d$ represents an invertible element. From this it follows that all the letters here are invertible, but this is not something that the Adjan algorithm will discover. 
\end{example}  

In fact, it is proved in \cite{Narendran91} 
that the problem of whether a finitely presented special monoid is a group is undecidable. From this is follows that there is no algorithm which takes finite special monoid presentations as input and computes the minimal invertible pieces of the relator.   

It is natural to ask whether there is an algorithm which takes finite one-relator special inverse monoid $\Ipres{A}{r=1}$ as input and computes the minimal invertible pieces of the relator $r$. This is a question that was considered by Margolis, Meakin and Stephen in \cite{OHarePaper}. As already mentioned above, in that paper the following example is introduced 
\[
\OHare=\Ipres{a,b,c,d}{abcdacdadabbcdacd = 1}=\Ipres{a,b,c,d}{r = 1},
\]
called the O'Hare monoid. As explained above, one thing that makes this example significant is that the defining relator word clearly has no overlaps with itself and hence the Adjan overlap algorithm would terminate instantly, and return $r$ as a single invertible piece of itself. However, Margolis,  Meakin and Stephen \cite{OHarePaper}, using van Kampen diagrams and Stephen's procedure, succeed in showing that $r$ in fact has a finer decomposition, namely 
\[
r\equiv abcd\cdot 
acd\cdot
ad\cdot
abbcd\cdot
acd.
\]
They provide no further information about the group of units of the monoid $\OHare$.

Our aim now is to give a unit computing algorithm for special inverse monoid presentations which improves on the Adjan algorithm. Specifically our new algorithm, which we shall call the \emph{Benois algorithm}, will have the following properties
\begin{enumerate}
\item it correctly computes the minimal invertible pieces of the O'Hare example, and 
\item it always computes a refinement of the Adjan algorithm.
\end{enumerate}
These properties mean that every invertible piece that the Adjan algorithm discovers, the Benois algorithm also discovers, and there are examples where the Benois algorithm discovers strictly smaller pieces than the Adjan algorithm does.  

Let $M = \Ipres{A}{r_i=1\ (i\in I)}$ be a finitely presented special inverse monoid. Let 
\[
\Sigma = \bigcup_{i \in I}(\pref(r_i) \cup \pref(r_i^{-1}))
\subseteq \overline{A}^*.
\]
Let $V = \mathrm{red}(\Sigma^*) \subseteq F_A$.  For each $i \in I$ decompose
\begin{equation}
\label{eq:dia1}
r_i \equiv r_{i 1} r_{i 2} \ldots r_{i {k_i}} 
\end{equation}
such that for every proper prefix $p$ of $r_i$ we have 
\begin{equation}
\label{eq:dia2}
\mathrm{red}(p^{-1}) \in V 
\Leftrightarrow
p \equiv r_{i 1} r_{i 2} \ldots r_{i j} \ \mbox{for some} \
j \in \{ 1, \ldots, k_i \}.
\end{equation}
Intuitively, this decomposition is obtained by `reading' $r_i$ from left to right, and `marking' each prefix $p$ with the property that $\mathrm{red}(p^{-1})\in V$.

\begin{lem}\label{lem:benois:pieces}
  For each $i \in I$, the factorisation \eqref{eq:dia1} is a decomposition into invertible pieces.  
\end{lem}
\begin{proof}
We need to prove that each $r_{ij}$ represents an invertible element of $M$.
This is equivalent to each prefix $p_{ij}\equiv r_{i1}r_{i2}\dots r_{ij}$ being invertible, and we proceed to prove this latter assertion.
Since every word in $\Sigma$ clearly represents a right invertible element of $M$ it follows that every word in $\Sigma^*$ represents a right invertible element of $M$. Applying Lemma~\ref{lem_good}(iii) it follows that every word in $V$ represents a right invertible element of $M$. 
Since $p_{ij}$ represents a right invertible element of $M$ it follows from Lemma~\ref{lem_good}(iii) that $p _{ij}= \red(p_{ij})$ in $M$. 
The word $p_{ij}^{-1}$ represents a left invertible element of $M$, and hence $p_{ij}^{-1} = \red(p_{ij}^{-1})$ also holds in $M$ 
by an obvious dual of Lemma~\ref{lem_good}(iii). 
By \eqref{eq:dia2} we have $\red(p_{ij}^{-1}) \in V$, implying that $p_{ij}^{-1}$ represents a right invertible element of $M$, and therefore $p_{ij}$ represents an invertible element of $M$, as required.
\end{proof}

Note that this lemma does not claim that \eqref{eq:dia1} is a decomposition into \emph{minimal} invertible pieces. In fact, it remains an open question whether this is the case. 

Next we turn our attention to the question of whether the above decomposition of the relations into invertible pieces is computable. It is a consequence of a theorem of Benois that the free group $F_A$ has decidable submonoid membership problem; see \cite{Lohrey15}.

\begin{lem}\label{lem:computing:benois}
There is an algorithm which takes any finite inverse monoid presentation  $\Ipres{A}{r_i = 1 \; (i \in I)}$ and any word $w \in \overline{A}^*$, and decides whether $\mathrm{red}(w) \in V$. 
  Therefore there is an algorithm which takes any finite inverse monoid presentation  $\Ipres{A}{r_i = 1 \; (i \in I)}$ and computes the decomposition into invertible pieces given by \eqref{eq:dia1} for all relators $r_i$ in the presentation.
\end{lem} 

\begin{proof}
 By definition the set $V$ is a finitely generated submonoid of the free group $F_A$, and the word $\red(w)$ is an element of the free group $F_A$. Hence Benois' Theorem can be applied to decide whether or not $\red(w) \in V$.  The second assertion in the statement of the lemma is then immediate. 
\end{proof}

We shall call the algorithm for computing invertible pieces given in Lemma~\ref{lem:computing:benois} the \emph{Benois algorithm}. 

Our aim is to prove the following result which explains the relationship between the Benois algorithm and the Adjan algorithm.    

\begin{thm}\label{thm_findsmore}
Let $M = \Ipres{A}{r_i=1 (i \in I)}$. Then for all $i \in I$ the decomposition of $r_i$ into pieces computed by the Benois algorithm is a refinement of the decomposition computed by the Adjan algorithm. 
\end{thm}

We need a few more definitions and lemmas before we can prove Theorem~\ref{thm_findsmore}. 
For any two words $\alpha, \beta \in \overline{A}^*$ we write $\alpha \rightarrow \beta$ if $\beta$ can be obtained from $\alpha$ by a single application of a rewrite rule $xx^{-1} \rightarrow 1$ or $x^{-1}x \rightarrow 1$ (with $x \in A$). Let us write $\rightarrow^*$ for the reflexive transitive closure of $\rightarrow$ i.e. $\alpha \rightarrow^* \beta$ means that $\beta$ can be obtained from $\alpha$ by a finite (possibly empty) sequence of deletions $xx^{-1} \rightarrow 1$ or $x^{-1}x \rightarrow 1$. If $\alpha \rightarrow^* \beta$ we say that $\alpha$ can be \emph{partially freely reduced} to $\beta$ (partially since $\beta$ need not be a reduced word).           

Define 
\[
\Sclo = \{ \beta \in \overline{A}^* \::\:
\alpha \rightarrow^* \beta \;
\mbox{for some } 
\alpha \in \Sigma^* \}. 
\]
So $\Sclo$ is the closure of the set $\Sigma^*$ under partial free reduction $\rightarrow^*$.    

It follows from Lemma~\ref{lem_good} that every word in $\Sclo$ represents a right invertible element of $M$.

\begin{remark}
Although we do not need it here, it may be shown that $\Sclo$ is a regular language, i.e. there is a finite state automaton $\mathcal{A}$ such that the language $L(\mathcal{A})$ recognised by the automaton is $L(\mathcal{A}) = \Sclo$.       
The proof of this can be found in \cite{Lohrey15}. 
The idea is to use an automaton saturation procedure.  
We begin with a finite state automaton $\mathcal{B}$ with $L(\mathcal{B})=\Sigma^*$ and then add $\epsilon$ transitions wherever we can read $xx^{-1}$, and repeat this process until no more $\epsilon$ transitions need to be added. 
Note that given  reduced word $u$ we can check if $u \in V$ by testing whether $u \in L(\mathcal{A})$.    
\end{remark}

Define $\U \subseteq \Sclo$ in the following way 
\[
\U = \{ \beta \in \Sclo \::\:  \beta^{-1} \in \Sclo \}. 
\] 
Note that every word $u \in \U$ represents an invertible element of $M$.   

The important thing to note about the set $\U$ is the following: 
Let $p$ be a prefix of some relator $r_i$ with $i \in I$. Then by definition $p \in \Sigma \subseteq \Sclo$. If in addition $p \in \U$ then $p$ is an invertible prefix of a relator and we will prove below that this invertible prefix $p$ will be computed by the Benois algorithm. 
However it is very important to note that the converse need not be true since it is possible that 
$p \in \Sclo$, $p^{-1} \not\in \Sclo$ but $\red(p^{-1}) \in \Sclo$.    
 
\begin{lem}\label{lem:UandBenois}
Let $p$ be a prefix of the relator $r_i$ where $i \in I$. If $p \in \U$ then $p$ is an invertible prefix of $r_i$ computed by the Benois algorithm.          
\end{lem}
\begin{proof}
Since $p \in \U$ it follows by definition of $\U$ that $p^{-1} \in \Sclo$. 
It follows that $\red(p^{-1}) \in \Sclo$ since by definition the set $\Sclo$ is closed under the operation of partial free reduction of words, 
and hence  
\[
\red(p^{-1}) \in \red(\Sclo) = \red(\Sigma^*) = V.
\]
It follows from the definition of the Benois algorithm that  $p$ is an invertible piece of $r_i$ computed by the Benois algorithm.   
\end{proof}
 
\begin{remark}
It is important to note that we are \emph{not} claiming that $\U$ contains \emph{all} the invertible prefixes of relators computed by the Benois algorithm. 
\end{remark}

We will now prove that $\U$ does contains \emph{all} prefixes of relators computed by the Adjan algorithm. Combining this with Lemma~\ref{lem:UandBenois} will prove Theorem~\ref{thm_findsmore}. 

\begin{lem}\label{lem:split}
Let $\gamma, \delta \in \overline{A}^*$ with $\gamma \rightarrow \delta$. Then for any decomposition $\delta \equiv \delta_1 \delta_2$ there exists a decomposition $\gamma \equiv \gamma_1 \gamma_2$ such that $\gamma_1 \rightarrow^* \delta_1$ and $\gamma_2 \rightarrow^* \delta_2$.       
\end{lem}
\begin{proof}
Since $\gamma \rightarrow \delta$ we can write $\gamma \equiv \gamma' xx^{-1} \gamma''$ where $\delta \equiv \gamma' \gamma''$ and $x \in \overline{A}$. 

Now we have $\delta_1 \delta_2 \equiv \delta \equiv \gamma' \gamma''$.     
Then $\gamma'$ is a prefix of $\delta_1$, or else $\gamma^{\prime\prime}$ is a suffix of $\delta_2$.
We consider the former case, and the latter is dealt with analogously.
Write $\delta_1 \equiv \gamma' \mu$. 
So $\gamma'\gamma'' \equiv \delta \equiv \delta_1 \delta_2 \equiv \gamma' \mu \delta_2$. Hence $\gamma'' \equiv \mu \delta_2$ and so $\gamma \equiv \gamma' xx^{-1} \mu \delta_2$. So if we set $\gamma_1 \equiv \gamma' xx^{-1} \mu$ and $\gamma_2 \equiv \delta_2$ then $\gamma \equiv \gamma_1 \gamma_2$ and $\gamma_1 \rightarrow^* \gamma'\mu \equiv \delta_1$ and $\gamma_2 \rightarrow^* \delta_2$, as required.
\end{proof}

\begin{cor}\label{lem:split:cor}
Let $\gamma, \delta \in \overline{A}^*$ with $\gamma \rightarrow^* \delta$. Then for any decomposition $\delta \equiv \delta_1 \delta_2$ there exists a decomposition $\gamma \equiv \gamma_1 \gamma_2$ such that $\gamma_1 \rightarrow^* \delta_1$ and $\gamma_2 \rightarrow^* \delta_2$.       
\end{cor}
\begin{proof}
Since $\gamma \rightarrow^* \delta$ we can write 
\[
\gamma \equiv w_0 \rightarrow w_1 \rightarrow \ldots \rightarrow w_k \equiv \delta. 
\] 
We have $w_k \equiv \delta \equiv \delta_1 \delta_2$. Then applying  
Lemma~\ref{lem:split} 
to $w_i \rightarrow w_{i+1}$ for all $i$ it follows that each $w_i$ can be decomposed as $w_i \equiv w_i' w_i''$ and 
\[
w_0' \rightarrow w_1' \rightarrow \ldots \rightarrow w_{k-1}' \rightarrow \delta_1
\quad\text{and}\quad
w_0'' \rightarrow w_1'' \rightarrow \ldots \rightarrow w_{k-1}'' \rightarrow \delta_2.
\] 
Setting $\gamma_1 \equiv w_0'$ and $\gamma_2 \equiv w_0''$ then proves the result.   
\end{proof}

Note that the set $\Sigma^*$ is clearly prefix closed.  

\begin{lem}\label{lem:pref:closed}
The set $\Sclo$ is prefix closed i.e. for any words $\alpha, \beta \in \overline{A}^*$ if $\alpha \in \Sclo$ and $\beta$ is a prefix of $\alpha$ then $\beta \in \Sclo$.       
\end{lem}
\begin{proof}
Let $\beta$ be a prefix of $\alpha$.    
Write $\alpha \equiv \beta \beta'$.  
Since $\alpha \in \Sclo$ there exists a word $\gamma \in \Sigma^*$ such that $\gamma \rightarrow^* \alpha \equiv \beta \beta'$.     
It follows from Corollary~\ref{lem:split:cor} that there exists a decomposition $\gamma \equiv \gamma_1 \gamma_2$ with $\gamma_1 \rightarrow^* \beta$ and $\gamma_2 \rightarrow^* \beta'$.
Since $\gamma \in \Sigma^*$ and $\gamma_1$ is a prefix of $\gamma$ it follows that $\gamma_1 \in \Sigma^*$.         
Since $\gamma_1 \in \Sigma^*$ and $\gamma_1 \rightarrow^* \beta$ it follows that $\beta \in \Sclo$.  
\end{proof}

Note that $\Sigma^*$ is closed under products.  

\begin{lem}\label{lem:closed:for:products}
The set $\Sclo$ is closed under products 
i.e. if $\alpha \in \Sclo$ and $\beta \in \Sclo$ then $\alpha\beta \in \Sclo$.    
\end{lem}
\begin{proof}
Let $\alpha_0, \beta_0 \in \Sigma^*$ such that 
$\alpha_0 \rightarrow^* \alpha$ and 
$\beta_0 \rightarrow^* \beta$. 
It follows that $\alpha_0 \beta_0 \in \Sigma^*$ with $\alpha_0\beta_0 \rightarrow \alpha\beta$ hence $\alpha\beta \in \Sclo$.         
\end{proof}

\begin{lem}\label{lem:relators:in:U}
We have $r_i \in \U$ for all $i \in I$.   
\end{lem}
\begin{proof}
By definition of $\Sigma$ both $r_i \in \Sigma$ and $r_i^{-1} \in \Sigma$. Since $\Sigma$ is a subset of $\Sclo$ it the follows 
from the definition of $\U$ 
that $r_i \in \U$.         
\end{proof}

The following is immediate from the definition of $\U$.   

\begin{lem}\label{lem:closed:under:inverses}
The set $\U$ is closed under inverses   
i.e. if $\alpha \in \U$ then $\alpha^{-1} \in \U$.   
\end{lem}

\begin{lem}\label{lem:Uclosed:under:products}
The set $\U$ is closed under products  
i.e. if $\alpha \in \U$ and $\beta \in \U$ then $\alpha\beta \in \U$.    
\end{lem}
\begin{proof}
Since $\U$ is closed under inverses by Lemma~\ref{lem:closed:under:inverses} it follows that $\alpha, \alpha^{-1}, \beta, \beta^{-1} \in \Sclo$.
But $\Sclo$ is closed under products by Lemma~\ref{lem:closed:for:products}, hence $\alpha\beta \in \Sclo$ and $(\alpha\beta)^{-1} \equiv \beta^{-1} \alpha^{-1} \in \Sclo$ hence $\alpha\beta \in \U$ by definition of $\U$.        
\end{proof}

Note that the set $\U$ will not in general be closed under prefixes.  

\begin{lem}\label{lem:closed:for:pfreductions}
The set $\U$ is closed under partial free reductions 
i.e.  
if $\alpha \in \U$ and $\alpha \rightarrow^* \beta$ then $\beta \in \U$.    
\end{lem}
\begin{proof}
Since $\alpha \in \U$ it follows from the definition of $\U$ that 
$\alpha \in \Sclo$ and $\alpha^{-1} \in \Sclo$.      
Since $\alpha \rightarrow^* \beta$ it follows from the definition of $\rightarrow^*$ that $\alpha^{-1} \rightarrow^* \beta^{-1}$.    
Hence $\beta \in \Sclo$ and $\beta^{-1} \in \Sclo$.
Since $\beta \in \Sclo$ and $\beta^{-1} \in \Sclo$ it follows from the definition of $\U$ that $\beta \in \U$.   
\end{proof}

We are now in a position to prove the key lemma that we need to prove our main result relating the Benois algorithm and the Adjan algorithm. 

\begin{lem}\label{lem:key:lemma:BA} 
The set $\U$ is closed under the Adjan overlap algorithm 
i.e. for all words $\alpha, \beta, \gamma \in \overline{A}^*$, if $\alpha\beta \in \U$ and $\beta\gamma \in \U$ then 
$\{\alpha, \beta, \gamma\} \subseteq \U$. 
In other words, the submonoid $\U$ of $\overline{A}^*$ has property ($\A$).   
  \end{lem}
\begin{proof} 
Since $\alpha\beta, \beta\gamma \in \U$ and $\U$ is closed under inverses by Lemma~\ref{lem:closed:under:inverses} it follows that
\[
\{\alpha\beta, \; \beta^{-1}\alpha^{-1}, 
\; \beta\gamma, \; \gamma^{-1}\beta^{-1} \}
\subseteq \U \subseteq \Sclo. 
\]  
Since $\Sclo$ is prefix closed 
by Lemma~\ref{lem:pref:closed},
and $\beta^{-1}\alpha^{-1}, \beta\gamma \in \Sclo$,  
it follows that $\beta^{-1}, \beta \in \Sclo$ which implies $\beta \in \U$ and $\beta^{-1}
 \in \U$.       
Since $\U$ is closed for products by Lemma~\ref{lem:Uclosed:under:products}, and   
$\beta^{-1} \in \U$ and $\beta\gamma \in \U$, it follows that $\beta^{-1}\beta\gamma \in \U$, and hence $\gamma \in \U$ since $\U$ is closed under partial free reductions by Lemma~\ref{lem:closed:for:pfreductions}.   
   
Also, $\alpha\beta \in \U$ and $\beta^{-1} \in \U$ implies $\alpha \beta \beta^{-1} \in \U$ by Lemma~\ref{lem:Uclosed:under:products}
which thus $\alpha \in \U$ 
by Lemma~\ref{lem:closed:for:pfreductions}.        

This completes the proof that $\{\alpha, \beta, \gamma\} \subseteq \U$. 
  \end{proof}

We now have all we need to prove our main result of this section. 

\begin{proof}[Proof of Theorem~\ref{thm_findsmore}]
Let $p$ be some prefix of some relator $r_i$, with $i \in I$. Write $r_i \equiv pq$.
Suppose $p$ is an invertible prefix of $r_i$ computed by the Adjan algorithm. 
By definition of the Adjan algorithm, this means that $p$ belongs to the submonoid $U(R)$ of $\overline{A}^*$ generated by the finite biprefix code $B(R)$ computed by the Adjan algorithm, where $R = \{r_i \::\: i \in I \}$.  
Since $r_i \in \U$ for all $i \in I$ by Lemma~\ref{lem:relators:in:U} 
and $\U$ is closed under the Adjan algorithm by Lemma~\ref{lem:key:lemma:BA},
it follows that $\U$ is a submonoid of $\overline{A}^*$ with $R \subseteq \U$ and $\U$ has property ($\A$).   
By definition $U(R)$ is the smallest submonoid of $\overline{A}^*$ such that $R \subseteq U(R)$ and $U(R)$ satisfies ($\A$).    
Hence it follows that $U(R) \subseteq \U$.   
Since $p$ belongs to $U(R)$ by assumption, it follows that $p$ also belongs to $\U$.     
But then by Lemma~\ref{lem:UandBenois}, since $p \in \U$ it follows that $p$ is an invertible prefix of a relator (specifically of the relator $r_i$) computed by the Benois algorithm.    
\end{proof}

\begin{defn}
We call a set of words $W \subseteq X^*$ \emph{overlap free} if for all $w_1, w_2 \in W$ we have 
\[
\mathrm{pref}(w_1)
\cap
\mathrm{suff}(w_2) = 
\begin{cases}
\{ \epsilon \} & \mbox{if $w_1 \neq w_2$} \\ 
\{ \epsilon, w_1 \} & \mbox{if $w_1 = w_2$}. 
\end{cases} 
\]
\end{defn} 

\begin{defn}
We call a set $W \subseteq \overline{A}^*$ \emph{$i$-overlap free} if $W \cup W^{-1}$ is overlap free. 
\end{defn} 

\begin{remark} 
It is not clear (and possibly not true in general) that the set of Benois pieces is $i$-overlap free. 
However, it is a consequence of the results above that if all the relators $r_i \; (i \in I)$ are reduced words then the Benois pieces are all $i$-overlap free.    
To see this, suppose that $p$ is a prefix of $r_i$ with $p \in V$ and 
$\red(p^{-1}) \in V$
(i.e. $p$ is computed by the Benois algorithm). 
Since $p$ is a reduced word, it follows that $p^{-1}$ is also reduced, so $p^{-1} \equiv \red(p^{-1}) \in V$. So we have both $p \in V$ and $p^{-1} \in V$ where $V = \red(\Sigma^*) \subseteq \Sclo$. 
Since $p \in \Sclo$ and $p^{-1} \in \Sclo$ it follows from the definition of $\U$ that $p \in \U$.  
It then follows by Lemma~\ref{lem:key:lemma:BA} that when all the relators $r_i$ are reduced words, all the pieces of relators computed by the Benois algorithm also belong to the set $\U$. But then since $\U$ is closed under taking inverses by Lemma~\ref{lem:closed:under:inverses} it follows from Lemma~\ref{lem:key:lemma:BA} that the set pieces of the relators computed by the Benois algorithm is $i$-overlap free. 
  \end{remark}

\begin{example}\label{ex:OHareBenois}
We now illustrate the Benois algorithm by applying it to the O'Hare monoid  
\[
\OHare=\Ipres{a,b,c,d}{abcdacdadabbcdacd = 1}=\Ipres{a,b,c,d}{r = 1}.
\]
We claim that when applied to this example the Benois algorithm will compute the following decomposition of $r$ into invertible pieces
\[
r\equiv abcd\cdot 
acd\cdot
ad\cdot
abbcd\cdot
acd.
\]
For notational convenience we set 
\[
\alpha \equiv abcd, \;
\beta \equiv acd, \; 
\gamma \equiv ad, \; 
\delta \equiv abbcd. 
\]
Using the same notation as above, we have 
\[
\Sigma = \mathrm{pref}(r) \cup \mathrm{suff}(r)^{-1} = 
\{ a, ab, abcd, abcda, \ldots, r \} \cup \{ d^{-1}, (cd)^{-1}, (acd)^{-1}, \ldots, r^{-1} \},  
\]
and $V = \mathrm{red}(\Sigma^*)$ which is a subset of $F_A$.  Since $r$ is a positive word, it follows that for every prefix $p$ of $r$, both $p$ and $p^{-1}$ are reduced words.  It follows that the Benois algorithm will compute the set of all prefixes $p$ of $r$ such that $p^{-1} \in V$.         

Let us now list some of the prefixes $p$ that this algorithm will compute.    
It is a straightforward calculation to verify that  
\[
\red( \;
\beta(\alpha \beta \gamma \delta \beta)^{-1} (\alpha \beta \gamma) 
\; )
\equiv \delta^{-1}
\]
and 
\[
\red(\beta^{-1} \alpha \delta^{-1}) \equiv \alpha^{-1}. 
\]

Then we have:
\begin{align*}
\alpha^{-1} &\equiv 
\red(\beta^{-1} \alpha \delta^{-1}  ) 
\equiv 
\red( \; 
\beta^{-1} \cdot 
\alpha 
\beta \cdot
(\alpha\beta\gamma\delta\beta)^{-1} \cdot
(\alpha\beta\gamma)
\; ) \in V \\  
(\alpha \beta)^{-1} &\equiv 
\beta^{-1} \alpha^{-1} \equiv
\red(\;
\beta^{-1} 
\beta^{-1} 
\alpha 
\delta^{-1} \;)   
\equiv
\red(\;
\beta^{-1} \cdot  
\beta^{-1} \cdot 
\alpha 
\beta \cdot
(\alpha\beta\gamma\delta\beta)^{-1} \cdot
(\alpha\beta\gamma)
\;) \in V  \\
(\alpha \beta \gamma)^{-1} &\equiv   
\red(
(\delta \beta)
(\alpha \beta \gamma \delta \beta)^{-1}
)
\equiv 
\red(
(\alpha \beta^{-1} \alpha) \beta 
(\alpha \beta \gamma \delta \beta)^{-1}
) \\
&\equiv 
\red(
\alpha \cdot \beta^{-1} \cdot \alpha \beta 
\cdot (\alpha \beta \gamma \delta \beta)^{-1}
) \in V \\
 (\alpha \beta \gamma \delta)^{-1} &=
 \red(\beta (\alpha \beta \gamma \delta \beta)^{-1}) 
=
\red(
\alpha \delta^{-1} \alpha (\alpha \beta \gamma \delta \beta)^{-1} ) \\ 
&=
\red(
(\alpha \beta) \cdot (\delta \beta)^{-1} \cdot  \alpha \cdot   
(\alpha \beta \gamma \delta \beta)^{-1} ) 
\in V.  
\end{align*}
This proves that the Benois algorithm computes a decomposition which is a refinement of the above decomposition. 
To show that it is exactly this decomposition that is computed it may be shown, by appealing to bicyclic monoid homomorphic images of this monoid, that none of the pieces of this decomposition has a proper non-empty invertible prefix.
Indeed, if there were a finer decomposition into minimal invertible pieces, then one quickly sees that in all possible cases this would imply that all the generators are invertible, and so the monoid $\OHare$ would need to be a group.  
However, $\OHare$ is not a group since the map from $\{a,b,c,d\}^*$ onto $\{x,y\}^*$ defined by $a \mapsto a$, $b \mapsto 1$, $c \mapsto 1$ and $d \mapsto d$ induces a surjective a homomorphism from $\OHare$ onto the bicyclic monoid $\Ipres{a,d}{ad=1}$.  

We leave the details of this as an exercise for the reader. 
\end{example}

\section{Examples}
\label{sec:Examples}

This section will contain some examples of one-relator groups, and one-relator inverse monoids, to which the results of the previous sections can be applied. 
We have already seen some applications in Section~\ref{sec:Makanin}, e.g. to the case where the pieces $p_i$ are powers of generators from $A$ (see \ref{it:F1} in Section~\ref{sec:Makanin}).   

We are motivated by the question of whether the group of units of a one-relator inverse monoid is a one-relator group. 
Recall that at the start of Section~\ref{sec:Makanin} we gave the example of the one-relator inverse monoid    
\[
M=\Ipres{a,b,c}{a (bc^2b^{-1}) a (bc^3b^{-1}) a (bc^2b^{-1}) a=1}.
\]
with minimal invertible pieces $a$, $bc^2b^{-1}$, and $bc^{3}b^{-1}$. 
We observed there that \[\Gpres{x,y,z}{xyxzxyx=1}=F_{x,y}\] is not a presentation for the group of units of the monoid $M$.  
A key feature of that example was that $\{a,bc^2b^{-1},bc^3b^{-1}\}$ is not a basis for $\Ggen{a,bc^2b^{-1},bc^3b^{-1}}\leq F_{a,b,c}$. 
On the other hand, $\{a, bcb^{-1} \}$ is a basis for $\Ggen{a,bc^2b^{-1},bc^3b^{-1}}\leq F_{a,b,c}$ 
and this basis satisfies the conditions of Theorem~\ref{thm:posets} (setting $\mu(a)=a$ and $\mu(bcb^{-1})=c$). Using this fact, we can express each of the original pieces $\{a,bc^2b^{-1},bc^3b^{-1}\}$ in terms of the basis $\{a, bcb^{-1} \}$, which rewrites the relator word  
\[
a (bc^2b^{-1}) a (bc^3b^{-1}) a (bc^2b^{-1}) a
\]
as
\[
a (bcb^{-1})^2 a (bcb^{-1})^3 a (bcb^{-1})^2 a
\]
and then by Corollary \ref{cor:ourMakanin} we can conclude that the group of units of $M$ is the one relator group with presentation 
\[
\Gpres{x,t}{xt^2xt^3xt^2x=1}. 
\]
It is natural to ask whether this approach might be used for other one-relator inverse monoids. 
As explained in the introduction, one of the key motivating examples for the work is the O'Hare monoid. We shall now show how the group of units of the O'Hare monoid can indeed be computed using a similar approach to the example above.

\subsection{The group of units of the O'Hare monoid}
\label{subsec:Ohare}
Recall from the introduction that the O'Hare monoid $\OHare$ is the inverse monoid defined by the following inverse monoid presentation  
\[
\Ipres{a,b,c,d}{abcdacdadabbcdacd = 1}=\Ipres{a,b,c,d}{r = 1}. 
\]
As explained in the introduction, this was one of the main motivating examples which prompted the research presented in this paper.  
As already mentioned above, 
Margolis, Meakin and Stephen \cite{OHarePaper}  proved that the decomposition into minimal invertible pieces of this defining relator is  
\[
r\equiv abcd\cdot 
acd\cdot
ad\cdot
abbcd\cdot
acd.
\]
They prove this using van Kampen diagrams and Stephen's procedure. We gave an alternative proof of this fact in Example~\ref{ex:OHareBenois} using the Benois algorithm (see Lemma~\ref{lem:computing:benois}). As in Example~\ref{ex:OHareBenois} for notational convenience we set 
\[
\alpha \equiv abcd, \;
\beta \equiv acd, \; 
\gamma \equiv ad, \; 
\delta \equiv abbcd. 
\]
We shall now show how the results from Section~\ref{sec:Makanin} can be applied to prove that the group of units $U(\OHare)$ of the O'Hare monoid $\OHare$ is the free group of rank $2$.

This is interesting for two reasons.  
Firstly, it shows that the groups of units of the O'Hare monoid $\OHare$ is a one-relator group (since every free group is vacuously a one-relator group). 
Secondly while the group of units of $\OHare$ is a one-relator group, the group of units is \emph{not} isomorphic to the group 
\[
H = \Gpres{\alpha, \beta, \gamma, \delta}{
\alpha \beta \gamma \delta \beta = 1},  
\]  
which is the presentation obtained by replacing each minimal invertible piece by a letter in the obvious way. 
Indeed, in the presentation for $H$ we can remove the redundant generator $\delta$, and so $H$ is the free group on 
$\{\alpha,\beta,\gamma\}$.

This gives another example (in addition to the example discussed at the start of Section~\ref{sec:Makanin}) showing that Makanin's Theorem for special monoids (Theorem~\ref{thm:SpecialMakanin}) does not generalise directly to special inverse monoids. 
We will construct further examples later in this paper which show just how dramatically Makanin's Theorem fails to generalise to special inverse monoids. 

So we are left with the task of proving that the group of units $U(\OHare)$ of $\OHare$ is a free group of rank two.   

The first step will be to write down an alternative one-relator presentation for $\OHare$ which has minimal invertible pieces that satisfy condition \ref{it:F3} from Section~\ref{sec:Makanin}.   

Since $\alpha$, $\beta$, $\gamma$ and $\delta$ all represent invertible elements of the monoid $\OHare$ it follows that  
\[
\delta \alpha^{-1} = abbcd d^{-1} c^{-1} b^{-1} a^{-1},  
\]
and
\[
\beta \gamma^{-1} = acd d^{-1} a^{-1} 
\]
are both invertible in $\OHare$. Hence by Lemma~\ref{lem_good} it follows that 
these words are equal in $\OHare$ to the words obtained by freely reducing them, namely the words $aba^{-1}$ and $aca^{-1}$. This shows that  
\[
aba^{-1}, \ aca^{-1} \ \mbox{and} \ ad
\] 
are all invertible in $\OHare$.  
It follows that any product of these words is also invertible in $\OHare$. In particular the word  
\[
u \equiv 
(aba^{-1}) 
(aca^{-1})
(ad)
(aca^{-1})
(ad)
(ad)
(aba^{-1})
(aba^{-1})
(aca^{-1})
(ad)
(aca^{-1})
(ad)
\]
is invertible in $\OHare$. Since this word is invertible it follows from Lemma~\ref{lem_good} that it is equal in $\OHare$ to the word obtained by freely reducing this word.    
But it is easy to see that the free reduction of this word gives the defining relator $r$ in the original presentation for $\OHare$.    
Hence $u = r = 1$ holds in the monoid $\OHare$.  

Conversely, let $M$ be the inverse monoid defined by the presentation 
\[ 
\Ipres{a,b,c,d}{
u=1} 
\]
where $u$ is the word above.  
Since $u$ is invertible in $M$ it follows that $u = \mathrm{red}(u)$ in $M$. Hence $u=r$ in $M$. 

Combining these observations we have proved 
\[
\Ipres{a,b,c,d}{r=1} = 
\Ipres{a,b,c,d}{r=1, u=1} = 
\Ipres{a,b,c,d}{u=1}.   
\]   
So the O'Hare monoid is defined by the presentation 
\[
\Ipres{a,b,c,d}{
(aba^{-1}) 
(aca^{-1})
(ad)
(aca^{-1})
(ad)
(ad)
(aba^{-1})
(aba^{-1})
(aca^{-1})
(ad)
(aca^{-1})
(ad)
=1
}.
\] 
Next we claim that working with this new presentation for $\OHare$ the minimal invertible pieces of the defining relator $u$ are $aba^{-1}$, $aca^{-1}$ and $ad$. 
We already proved above that each of these words represents an invertible elements of $\OHare$, so these are certainly invertible pieces of the defining relator $u$.    
For minimality, if any of these pieces was not a minimal invertible piece then in all cases it would follow that $a$ is invertible. But $a$ is not invertible in $\OHare$ since   
$ad$ is a minimal invertible piece of $r$.

We claim that the decomposition of $u$ into minimal invertible pieces in  the presentation $\Ipres{a,b,c,d}{u=1}$  
satisfies condition \ref{it:F3} from Section~\ref{sec:Makanin}. 
For this we just need to observe that each piece contains a letter which appears in that piece exactly once, and does not appear in any of the other pieces. Here we can take $b$ in the piece $aba^{-1}$, $c$ in the piece $aca^{-1}$, and the letter $d$ in the piece $ad$.       

Since condition \ref{it:F3} holds it follows that the hypotheses of Corollary~\ref{cor:ourMakanin} are satisfied and hence applying this corollary we conclude that the group of units $U(\OHare)$ of $\OHare$ is isomorphic to 
\[
\Gpres{x,y,z}{xyzyzzxxyzyz=1}.
\]

Comparing this with the example at the start of Section~\ref{sec:Makanin} what we have proved here is that the situation with the O'Hare monoid is similar. 
Indeed, the original set of pieces $\{abcd, acd, ad, abbcd \}$ is not a basis for $\Ggen{abcd, acd, ad, abbcd } \leq F_{a,b,c,d}$. 
On the other hand, $\{aba^{-1}, aca^{-1}, ad \}$ is a basis for    
$\Ggen{abcd, acd, ad, abbcd } \leq F_{a,b,c,d}$ 
and this basis satisfies the conditions of Theorem~\ref{thm:posets} (setting $\mu(aba^{-1})=b$, $\mu(aca^{-1})=c$ and $\mu(ad)=d$).
If we rewrite the pieces in the original O'Hare presentation in terms of this basis and then apply Corollary~\ref{cor:ourMakanin} we conclude (via the calculations above) that the group of units of $\OHare$ is isomorphic to the one-relator group $\Gpres{x,y,z}{xyzyzzxxyzyz=1}$.

Now we have
\begin{eqnarray*}
\Gpres{x,y,z}{xyzyzzxxyzyz=1} & = & \Gpres{x,y,z,t}{xttzxxtt=1, t=yz} \\
& = & \Gpres{x,y,z,t}{xttzxxtt=1, y=tz^{-1}} \\
& = & \Gpres{x,z,t}{xttzxxtt=1} \cong \mathrm{FG}(x,t).  
\end{eqnarray*}
The last isomorphism follows upon removing the redundant generator $z$ in the previous presentation.
This completes the proof that the group of units of the O'Hare monoid is a free group of rank $2$.

Both the O'Hare monoid and the example at the start of Section~\ref{sec:Makanin} are examples of one-relator inverse monoids 
\[
\Ipres{A}{r=1}
\]  
which show that the obvious generalisation of Makanin's theorem to special one-relator monoids does not hold. 
However, for both these examples it was possible to resolve the issue by finding a suitable basis for the subgroup of $F_A$ generated by the pieces of $r$ and then rewriting each of the pieces, and the relator $r$, in terms of this new basis. 
In both cases this gave us a one-relator presentation for the group of units of the inverse monoid in question. 
Specifically, these examples do not resolve the question of whether the group of units of a special one-relator inverse monoid is a one-relator group.    

Given these examples it is natural to ask whether in fact the key property that we need for a set of pieces to be free for substitutions is that they are a basis for the subgroup of $F_A$ that they generate (like in these two examples).    
The following example which was originally due to Higman shows that this is not the case. 
Specifically, it shows that the conditions in the theorems in Section~\ref{sec:Makanin} cannot be weakened to just insisting that the set of pieces is Nielsen reduced.

\subsection{The G. Higman example}\label{subsection_Higman}

This example appears in the following paper of Steve Pride 
\cite{Pride76}
where he attributes the example to Graham Higman. As explained there, if we let 
\[
B = \Gpres{a,b}{b^{-1} a^2 b = a^3}
\] 
then the subgroup $\Ggen{a^4,b} \leq B$ is not free and every presentation of the group $\Ggen{a^4,b}$ with respect to the generating set $\{a^4, b\}$ requires at least two relators. Note the group $B$ is the well-known Baumslag-Solitar group $BS(2,3)$.   

This example arose as part of an investigation by Pride of conditions under which subgroups of one-relator groups are again one-relator. 
One interesting thing in this example is that the two-generated subgroup 
$\Ggen{a^4,b}$ 
of the one-relator group $B$ does not admit a one-relator presentation with respect to the generators $\{a^4, b\}$ (as proved by Higman). 
On the other hand, the subgroup $\Ggen{a^4,b}$ is actually a one-relator group. 
Indeed it may be shown that $\Ggen{a^4,b} = B$ that is, $\{a^4,b\}$ is a generating set for $B$. 
To see this it suffices to observe that in $B$ we have 
\[
a^2 =
[a^6] (a^4)^{-1} = 
[(b^{-1} a^2 b)^2] (a^4)^{-1}=
[b^{-1} a^4 b](a^4)^{-1} \in  \Ggen{a^4,b}. 
\]
But now since $a^2$ and $b$ both belong to $\Ggen{a^4,b}$ it follows that $a^3 = b^{-1} a^2 b \in \Ggen{a^4,b}$, and finally since $a^3$ and $a^2$ both belong to $\Ggen{a^4,b}$ it follows that $a = (a^3)(a^2)^{-1}$ also belongs to $\Ggen{a^4,b}$. Hence  $\Ggen{a^4,b} = \Ggen{a,b} = B$. 
So this is a concrete example showing that a $2$-generated one-relator group can admit a one-relator presentation with respect to a generating set of size two, but not admit any one-relator presentation with respect to another generating set of size two. Among other things this shows how sensitive the property of being one-relator is to the choice of finite generating set for the group.  

The following result shows how this example can be adapted to give an example of a Nielsen reduced set which is not free for substitutions into one-relator groups. 
We shall not need the definition of Nielsen reduced here. 
The definition can be found in \cite[page~6]{lyndonschupp}.  
The thing that is relevant for us here is the result 
(see \cite[Proposition~2.6]{lyndonschupp}) that if 
$X \subseteq F_A$ is Nielsen reduced then $\Ggen{X} \leq F_A$ is free with basis $X$.

\begin{prop}\label{thm_Higman}
Let $X = \{ ab^{-1}a^2, b, a^{4} \}$, which is Nielsen reduced a subset of the free group $F_{a,b}$. Let  
$\phi : F_{x,y,z} \rightarrow F_{a,b}$ be the unique homomorphism extending
\[
x \mapsto ab^{-1}a^2, \quad
y \mapsto b, \quad
z \mapsto a^{4}. 
\]
Set 
\[
K = \Gpres{x,y,z}{xyz^{-1}=1}
\]
and
\[
G =  
\Gpres{a,b}{(ab^{-1}a^2)(b)(a^{-4})=1}.
\]
Let $\hat{\phi}: K \rightarrow G$ be the homomorphism induced by $\phi$. Then $\hat{\phi}$ is not injective, and hence is not an isomorphism between $K$ and the subgroup $H$ of $G$ generated by the set $X = \{ ab^{-1}a^2, b, a^{4} \}$.
That is, the Nielsen reduced set $X = \{ ab^{-1}a^2, b, a^{4} \}$ is not free for substitution into the one-relator group
$\Gpres{x,y,z}{xyz^{-1}=1}$ 
via
$x \mapsto ab^{-1}a^2, y \mapsto b$ and $z \mapsto a^{4}$. 
\end{prop}
\begin{proof}
It is a straightforward exercise to check that $X$ is a Nielsen reduced set of reduced words in the free group $F_{a,b}$, 
and hence $\Ggen{X} \leq F_{a,b}$ is freely generated by $X$.    
Certainly $\phi$ defines a homomorphism since $\phi(xyz^{-1})=1$ in $G$. 
Clearly, $K$ is isomorphic to the free group on $y$ and $z$. 
In particular the words $z^2 y$ and $yz^3$ represent distinct elements in the group $K$. On the other hand 
\[
\phi(z^2 y) = \phi(y z^{3}) 
\]
in $G$. 
Indeed, in $G$ we have $a^2b=ba^3$, and hence
\[
\phi(z^2y)=a^8b=(a^2)^4b=b(a^3)^4=ba^{12}=\phi(yz^3),
\]
as required.
\end{proof}

Note that as it stands the example in Proposition~\ref{thm_Higman} does not immediately give a special inverse monoid with interesting properties, since the set  
$X = \{ ab^{-1}a^2, b, a^{4} \}$ 
cannot be the set of minimal invertible words in an inverse monoid presentation (since for instance $a^4$ invertible implies $a$ is invertible and hence $a^4$ is not a minimal invertible piece).    
Also, in this example the subgroup generated by the pieces 
$X = \{ ab^{-1}a^2, b, a^{4} \}$
is in fact a one-relator group since (as explained before the statement of the theorem), but this theorem shows that it does not admit the ``obvious'' one relator presentation given by replacing the pieces by letters in the obvious way.   

However, this example is the first evidence that there may exist examples of one-relator special inverse monoids with group of units not being a one-relator group.  This will be explored further in the next section.

\section{A construction}
\label{sec:Construction}

In this section we shall give a general construction which will then be used in the next section to give examples of special inverse monoids that exhibit unexpected behaviour in terms of their groups of units.

The construction we give here was used in \cite{gray:undecidable} to give an example of a one-relator special inverse monoid with undecidable word problem.  
It has also been used in the papers \cite{GrayDolinka2021} and \cite{GraySilvaSzakacs} to construct counterexamples to questions about the prefix membership problem for one-relator groups, and the word problem for finitely presented inverse monoids with hyperbolic Sch\"{u}tzenberger graphs.  
To make this paper self contained we provide full details of the construction here. 
Also, we use slightly different notational conventions than those used in \cite{gray:undecidable}. 
Some of the results we need about this construction have already been proved in \cite{gray:undecidable}, but there are some other facts about the construction that we need here that are not proved in that paper, so we will provide the necessary proofs.  

\subsection{The construction}

Let $A=\{a_1,\dots,a_n\}$ be a finite non-empty alphabet. Let $Q = \{ r_i \::\: i \in I \}$ be a subset of $\overline{A}^*$, where $I$ is non-empty and we assume there is a distinguished symbol $1 \in I$. 
Let $W = \{ w_j \::\: j=1,\dots, k \}$ be a finite subset of $\overline{A}^*$. Let $t$ be a symbol not in $A$. 

We define two groups which depend on $Q$ by
\begin{align}
\label{eq:KQ}
K_Q &= \Gpres{A}{r_i=1 \; (i \in I)} \\
\label{eq:GQ}
G_Q &= \Gpres{A, t}{r_i=1 \; (i \in I)} = K_Q \ast F_t. 
\end{align} 

Given a list of words $u_1, \ldots, u_m \in \overline{A}^*$ we define  
\[
e(u_1, u_2, \ldots, u_m) = u_1 u_1^{-1} u_2 u_2^{-1} \ldots u_m u_m^{-1},
\]
noting that this clearly freely reduces to $1$ in the free group $F_A$. 

We also define two inverse monoids, the first depending on $Q$, and the second on both $Q$ and $W$, as follows:
\begin{align}
\label{eq:NQ}
N_Q &= \Ipres{A}{r_i=1 \; (i \in I)} \\
\label{eq:MQW}
M_{Q,W} &= \Ipres{A, t}{fr_1=1, 
r_i=1 \; (i \in I \setminus \{ 1 \})} 
\end{align}
where 
\begin{equation}
\label{eq:fdef}
f = 
e(a_1, \ldots, a_n,  
tw_1t^{-1}, \ldots, tw_kt^{-1}, 
a_1^{-1}, \ldots, a_n^{-1}) .  
\end{equation}

We make the following observations about the groups and monoids we have just defined:
\begin{enumerate}[leftmargin=10mm,itemsep=1mm,label=\textup{(M\arabic*)}]
\item \label{it:maxim}
$K_Q$ is the maximal group image of $N_Q$, and $G_Q$ is the maximal group image of $M_{Q,W}$
(under the natural homomorphisms).
\item \label{it:altpres}
Presentation \eqref{eq:MQW} for $M_{Q,W}$ is equivalent 
to
\begin{equation}
\label{eq:MQWalt}
M_{Q,W}= \Ipres{A, t}{ 
r_i=1 \ 
 a_j a_j^{-1} = a_j^{-1} a_j = 1, \ 
 t w_l t^{-1} t w_l^{-1} t^{-1} = 1 \ (i\in I,\ j \in [n],\ l\in [k])  
}
\end{equation}
in the sense that two words over $\overline{(A \cup \{t\})}$ are equal modulo the relations of 
\eqref{eq:MQW} if and only if they are equal modulo the relations \eqref{eq:MQWalt}.  
\item \label{it:ainv}
All the generators $a_j \in A$ are invertible in $M_{Q,W}$, while $t$ and all of the elements $t w_l t^{-1}$ are right invertible in $M$. 
\item \label{it:Bhomi}
$M_{Q,W}$ maps naturally onto the bicyclic monoid $B=\Ipres{b}{bb^{-1}=1}$ via
$a_j\mapsto 1$, $t\mapsto b$.
\item \label{it:tnoni}
$t$ is not invertible in $M_{Q,W}$.
\end{enumerate}

Indeed, observation  
\ref{it:maxim} 
is immediate from the definitions and the general fact that the maximal group image of an inverse monoid defined by a presentation is the group defined by the same presentation.  
\ref{it:altpres}
can be proved using Lemma~\ref{lem_good}. 
A full proof is given in \cite[Lemma~3.3]{gray:undecidable}. 
\ref{it:ainv} follows from \ref{it:altpres} and the defining relations in the presentation \eqref{eq:MQWalt}. 
\ref{it:Bhomi} is immediate from the fact that $M_{Q,W}$ is given by the presentation \eqref{eq:MQWalt} and then the observation that replacing $a$ by $1$ and $t$ by $b$ in that presentation just leaves the relation $bb^{-1}bb^{-1}=1$ which certainly holds in $B$.   
\ref{it:tnoni} follows from \ref{it:Bhomi} since $b$ is not invertible in the bicyclic monoid $B$, and the image of an invertible element of a monoid under a surjective homomorphism must again be an invertible element.    

We will make use of the following general lemma, which follows immediately from the universal properties of free products.

\begin{lem}\label{lem_submon_free_prod}
Let $G$ and $H$ be groups. Let $M$ be a submonoid of $G$ and $N$ be a submonoid of $H$. Then the submonoid of the free product $G \ast H$ generated by $M \cup N$ is isomorphic to $M \ast N$. 
\end{lem} 

The following lemma is probably well known, but we give a proof for completeness. 

\begin{lem}\label{lem_HandConjH}
Let $H$ be a group. Then the subgroup of the free product $G = H \ast F_t$ generated by $H \cup tHt^{-1}$ is isomorphic to $H \ast H$. 
\end{lem}
\begin{proof}
Set $\overline{H} = t H t^{-1} \leq G = H * F_t$. 
Let $L$ be the subgroup of $G$ generated by $H \cup \overline{H}$. 
Now $H$ and $\overline{H}$ are subgroups of $L$ and $H \cap \overline{H} = \varnothing$ by the normal form theorem for elements in free products of groups. Given a reduced sequence $g_1, g_2, \ldots, g_n$ where the $g_i$ all belong to $L$ and alternate between $H$ and $\overline{H}$, if $n > 0$ then it again follows from the normal form theorem in $G$ that $g_1 g_2 \ldots g_n \neq 1$. Thus all the conditions of 
\cite[Lemma 1.7]{lyndonschupp} are satisfied, and we conclude that $L \cong H * \overline{H}$. Clearly $H \cong \overline{H}$ since they are conjugate subgroups of $G$. This completes the proof that $L \cong H * H$. 
\end{proof}

We are now in a position to state and prove the main result of this section which establishes the key properties of the monoid $M_{Q,W}$ that we will need in our applications and examples in the next section.  

\begin{thm}\label{thm_GeneralConstructionOneSided}
With the above definitions and notation, the special inverse monoid $M_{Q,W}$ has the following properties.
\begin{enumerate}[leftmargin=10mm,itemsep=1mm,label=\textup{(\roman*)}]
\item \label{it:constvi}
The presentation \eqref{eq:MQW} of $M_{Q,W}$ has the same number of defining relations as the original presentation \eqref{eq:KQ} of $K_Q$. 
\item \label{it:consti}
The monoid $M_{Q,W}$ is E-unitary.
\item \label{it:constii}
Let $T_W$ be the submonoid of the group $K_Q$ generated by $W$. Then the submonoid of right units $R(M_{Q,W})$ of $M_{Q,W}$ is isomorphic to the submonoid of the group $G_Q=K_Q \ast F_t$ generated by 
$\{ t\} \cup K_Q \cup t T_W t^{-1}$. 
\item \label{it:constiii}
Let $V_{Q,W}$ be the submonoid of $R(M_{Q,W})$ generated by 
$A \cup A^{-1} \cup tWt^{-1}$.  
Then $V_{Q,M}$ is isomorphic to the free product $K_Q \ast T_W$ and the complement of $V_{Q,W}$ in $R(M_{Q,W})$ is an ideal of $R(M_{Q,W})$.  
\item \label{it:constiv}
The group of units $U(M_{Q,W})$ of $M_{Q,W}$ is isomorphic to the free product 
$
K_Q \ast H_{W'} 
$
where $H_{W'}$ is the subgroup of $K_Q$ generated by the subset $W'$ of $W$ defined by
\[
W' = \{ w_l \in W: tw_lt^{-1} \ \mbox{is invertible in} \; M_{Q,W} \}. 
\]
In particular if $W = W^{-1}$ then $W'=W$ and so $U(M_{Q,W}) \cong K_Q \ast H_W$.   
\item \label{it:constv}
If $R(M_{Q,W})$ is finitely presented then  $K_Q$ and $T_W$ are both finitely presented. 
\end{enumerate} 
\end{thm}

\begin{proof}
\ref{it:constvi}
is obvious since they both have $|I|$ defining relations, and \ref{it:consti} is \cite[Theorem 3.8]{gray:undecidable}.
\medskip

\ref{it:constii} By \cite[Lemma~1.5]{margolis93}, since $M_{Q,W}$ is E-unitary, it follows that the canonical homomorphism $\theta: M_{Q,W} \rightarrow G_Q=K_Q\ast F_t$ from $M_{Q,W}$ onto its maximal group image
(see \ref{it:maxim}) induces an embedding of each $\mathscr{R}$-class of $M$ into $G_Q$. In particular $\theta$ induces an embedding of the right units $R(M_{Q,W})$ into $G_Q$. By the argument given in the proof of 
\cite[Proposition~4.2]{imm01} it follows that $R(M_{Q,W})$ is generated by the prefixes of the defining relators in the presentation \eqref{eq:MQWalt} of $M_{Q,W}$.  
Note that in \cite[Proposition~4.2]{imm01} an assumption is made that the defining relators are all cyclically reduced words, but this hypothesis is not used in the proof, and so the statement holds with that assumption removed (cf. Theorem \ref{IMM:gens}).

This implies that $R(M_{Q,W})$ is isomorphic to the submonoid of $G_Q$ generated by 
the prefixes of the defining relators in the presentation  
\eqref{eq:MQWalt}.
We claim that this is equal to the submonoid of $G_Q$ generated by the set 
\[
Y = A \cup A^{-1} \cup \{t \} \cup \{ tw_l t^{-1}: l \in [k] \}. 
\]
Indeed, we have already observed that all these words represent right invertible elements (see \ref{it:ainv}).
Furthermore, 
all the prefixes of relators in \eqref{eq:MQWalt} are equal 
in $M_{Q,W}$ 
to products of those words.
This is immediate for the prefixes of $r_i$, $a_ja_j^{-1}$, $a_j^{-1}a_j$, as they are products of generators from $A$.
The proper prefixes of $tw_lt^{-1}$ are products of generators from $A\cup\{t\}$.
Finally, longer prefixes of $tw_lt^{-1}tw_l^{-1}t^{-1}$ are equal 
in $M_{Q,W}$ 
to prefixes of $tw_lt^{-1}$ by 
Lemma \ref{lem_good}\ref{it:lg1}. 
This completes the proof of the claim that $R(M_{Q,W})$ is isomorphic to the submonoid of $G_Q$ generated by $Y$.   

Since the submonoid of $G_Q$ generated by $A \cup A^{-1}$ is $K_Q$, 
and the submonoid of $G_Q$
generated by $\{ tw_j t^{-1}: j \in [k] \}$ is $t T_W t^{-1}$, this completes the proof of \ref{it:constii}. 
\medskip

\ref{it:constiii} 
As in the previous part,  E-unitarity of $M_{Q,W}$ implies that $V_{Q,W}$ is isomorphic to the submonoid of
 $G_Q$ generated by $K_Q \cup t T_W t^{-1}$. 
 By Lemma~\ref{lem_HandConjH} the submonoid of $K_Q \ast F_t$ generated by $K_Q \cup t K_Q t^{-1}$ is isomorphic to $K_Q \ast K_Q$. Combining this fact with Lemma~\ref{lem_submon_free_prod} we conclude that the submonoid of $K_Q \ast F_t$ generated by $K_Q \cup t T_W t^{-1}$ is isomorphic to $K_Q \ast T_W$. 
Hence we have shown that $V_{Q,W}$ is isomorphic to $K_Q \ast T_W$. 

To see that $R(M_{Q,W}) \setminus V_{Q,W}$ is an ideal in $R(M_{Q,W})$ it suffices to observe that 
under the natural epimorphism from $M_{Q,W}$ onto the bicyclic monoid $B$
(see \ref{it:Bhomi}), the submonoid $R(M_{Q,W})$ 
is mapped onto
$R'=\{ b^i\::\: i\in \N_0\}$, while $R(M_{Q,W}) \setminus V_{Q,W}$ is mapped onto $V'=\{ b^i\::\: i\in \N\}$, and that
$V'$ is an ideal of $R'$. 
Since the preimage of an ideal, with respect to a surjective homomorphism, is itself an ideal this 
completes the proof that the complement of $V_{Q,W}$ in $R(M_{Q,W})$ is an ideal of $R(M_{Q,W})$.  
 
\medskip

\ref{it:constiv}
Observe that $U(M_{Q,W})$ is contained in $R(M_{Q,W})$ and the complement $R(M_{Q,W}) \setminus U(M_{Q,W})$ is an ideal of the monoid $R(M_{Q,W})$. 
Also, it follows from the proof of part (iii) that $R(M_{Q,W})$ is the submonoid of $M_{Q,W}$ generated by $\{t\} \cup A \cup A^{-1} \cup tWt^{-1}$. 
Since 
$R(M_{Q,W}) \setminus U(M_{Q,W})$ is an ideal of the monoid $R(M_{Q,W})$
it follows that  
\[
Z=(\{t\} \cup A \cup A^{-1} \cup tWt^{-1}) \cap U(M_{Q,W})
\]
is a monoid generating set for $U(M_{Q,W})$ 
(this follows from the more general fact \cite[Lemma~3.5]{gray:undecidable}). 
By \ref{it:tnoni}, $t$ is not invertible in $M_{Q,W}$ and hence $t \not\in Z$.  
By \ref{it:ainv}, $A \cup A^{-1} \subseteq Z$. 
The remaining elements of $Z$ are those from the set $tWt^{-1} \cap U(M_{Q,W})$ which is by definition equal to the set $tW't^{-1}$ where    
\[
W' = \{ w_l \in W: tw_lt^{-1} \ \mbox{is invertible in} \; M_{Q,W} \}. 
\]
This proves that the group of units $U(M_{Q,W})$ is isomorphic to the submonoid of   
$R(M_{Q,W})$ generated by 
\[
Z = A \cup A^{-1} \cup tW't^{-1}. 
\] 
It follows that the image of $U(M_{Q,W})$ under the isomorphism from \ref{it:constii} is the 
submonoid of $G_Q = K_Q\ast F_t$ generated by image of $Z$ under  
the isomorphism from \ref{it:constii}. 
Since the submonoid of $K_Q\ast F_t$ generated by $A \cup A^{-1}$ is $K_Q$   
it follows that 
the image of $U(M_{Q,W})$ under the isomorphism from \ref{it:constii}
is precisely the submonoid of 
$K_Q \ast F_t$ 
generated by $K_Q\cup tW't^{-1}$.
Since $U(M_{Q,W})$ is a group this must also be equal to the subgroup 
of 
$K_Q \ast F_t$ 
generated by $K_Q\cup tW't^{-1}$.
Under the isomorphism from \ref{it:constiii}, this corresponds to 
$K_Q\ast H_{W'}$, proving the first assertion.
For the second assertion observe that if $W=W^{-1}$ then $W'=W$ immediately from the presentation
\eqref{eq:MQWalt} for $M_{Q,W}$.
\medskip

\ref{it:constv} If $R(M_{Q,W})$ is finitely presented then by part \ref{it:constiii} it follows that $V_{Q,W}$ is finitely presented since its complement is an ideal. 
The fact that finite presentability is inherited by submonoids with ideal complement is well know e.g. it is a corollary of \cite[Theorem~B]{GRBoundaries2011}. 
Therefore $K_Q \ast T_W$ is finitely presented which implies that $K_Q$ and $T_W$ are both finitely presented being retracts of $K_Q \ast T_W$.
Here we have used the fact that any retract of a finitely presented monoid is again finitely presented. This follows from the proof of  \cite[Theorem~3.4]{WangPride2000} 
(see in particular the remark in  
\cite{WangPride2000} immediately after the proof of Theorem~3.4 of that paper).
\end{proof}

\section{Applications to further examples}
\label{sec:applications}

The central theme in this article as outlined in Section~\ref{sec:intro} has been to investigate the extent to which analogues of the results of Adjan/Makanin/Zhang for finitely presented special monoids can be proved for special inverse monoid presentations.  
Both the example at the start of Section~\ref{sec:Makanin} and the O'Hare monoid example considered in Subsection~\ref{subsec:Ohare} show that while the obvious na\"{i}ve generalisation to special inverse monoids does not hold, 
by an appropriate choice of basis, the hypotheses of Theorem~\ref{thm:Usandwiched} are satisfied by these two examples. In particular in each case their groups of units are one-relator groups. 
As explained in Section~\ref{sec:intro} it is natural to ask whether the assumptions of Theorem~\ref{thm:Usandwiched} are perhaps always satisfied, or at least always in the case of one-relator presentations.
The main results in Section~\ref{sec:Makanin} identify several sufficient conditions under which the assumptions of Theorem~\ref{thm:Usandwiched} are satisfied, see for example the conditions \ref{it:F1},  \ref{it:F2},  and \ref{it:F3}.  

In this section we will use the general construction and Theorem~\ref{thm_GeneralConstructionOneSided} from the previous section to prove that in general none of the main results of Adjan/Makanin/Zhang for special monoids can be extended to special inverse monoids. 

We are also prove a result which shows a close relationship between finite presentability of the units of special one-relator inverse monoids and the question of coherence of one-relator groups.

\subsection{A one-relator special inverse monoid whose group of units if not one-relator}

If we input a one-relator group $K_Q$ into the construction from Section~\ref{sec:Construction} then 
Theorem~\ref{thm_GeneralConstructionOneSided}\ref{it:constiv}
shows that the group of units of the one-relator special inverse monoid $M_{Q,W}$ is isomorphic to the free product of groups $K_Q \ast H_{W}$ where by appropriate choice of $W$ in the construction, the group 
$H_W'$ can be any finitely generated subgroup of $K_Q$ that we please. 
In particular the set $W$ can be chosen so that the group of units of 
the one-relator special inverse monoid 
$M_{Q,W}$ is isomorphic to $K_Q \ast K_Q$. 
Since in general the free product 
$K_Q \ast K_Q$ 
of a one-relator group $K_Q$ with itself is not a one-relator group, this can be used to construct examples of special one-relator inverse monoids whose groups of units are not one-relator. 
This leads to the following result. 

\begin{thm}\label{thm:notOneRelUnits}
There exists a one-relator special inverse monoid $M = \Ipres{A}{w=1}$ whose group of units $G$ is not a one-relator group with respect to any finite generating set. 
Moreover, $M$ can be chosen to be E-unitary.  
\end{thm}
\begin{proof}
Let $A = \{a,b,c,d\}$, let $Q = \{ aba^{-1}b^{-1}cdc^{-1}d^{-1} \}$ and set $W = \{a^{\pm 1},b^{\pm 1},c^{\pm 1},d^{\pm 1}\}$. 
Then by Theorem~\ref{thm_GeneralConstructionOneSided} \ref{it:constvi}, \ref{it:consti}, \ref{it:constiv}  the monoid
$M=M_{Q,W}$ is a one-relator E-unitary special inverse monoid 
with group of units $U(M) \cong K \ast K$ where
\[
K = K_Q = 
\Gpres{a,b,c,d}{[a,b][c,d]=1} =
\Gpres{a,b,c,d}{aba^{-1}b^{-1}cdc^{-1}d^{-1}=1}. 
\]
Note that $K$ is a torsion-free one-relator group. In fact it is a one-relator surface group: the fundamental group of a surface of genus $2$. To complete the proof it now suffices to prove that the free product $K \ast K$ of this one-relator group with itself is not a one-relator group with respect to any finite generating set. 
One way to see this is as follows.   

It follows from Lyndon's Identity Theorem \cite{lyndon50}, 
see also \cite[page~37, Example 3]{brown94}, 
that $H_2 (K) = \mathbb{Z}$, where $H_2(K)$ denotes the second homology group of $K$. 
Indeed, it follows from Lyndon's results that if $G = \Gpres{X}{s=1}$ is a    
torsion-free one relator group, where $s$ is a cyclically reduced word, then 
\[
H_2(G) \cong \begin{cases}
\mathbb{Z} & \mbox{if } s \in F_X' \\
0 & \mbox{otherwise,} 
\end{cases}
\]
where $F_X'$ is the derived subgroup of the free group. 
From $H_2(K) = \mathbb{Z}$ it follows by 
\cite[Corollary 6.2.10]{weibel94}
or 
\cite[page 220, Theorem~14.2]{HiltonAndStambach}
that 
\[
H_2(K \ast K) \cong 
H_2(K) \oplus H_2(K) \cong 
\mathbb{Z} \oplus \mathbb{Z}.
\]
Clearly $K \ast K$ is torsion free, since $K$ is torsion free, and thus by Lyndon's result above it follows that $K \ast K$ cannot be a one-relator group with respect to any generating set. 
\end{proof}

\subsection{A one-relator special inverse monoid whose 
group of units is finitely presented but 
the monoid of right units is not finitely presented}

An important structural result arising from the work of Adjan/Makanin/Zhang is that the monoid of right units of a finitely presented special monoid is isomorphic to a free product of its group of units and a finitely generated free monoid.  
The results in this subsection will show that this is far from being true in the case of special inverse monoids.  

Theorem~\ref{thm_GeneralConstructionOneSided}\ref{it:constv}
tells us that, 
for that construction,  
finite presentability of 
the right units $R(M_{Q,W})$ of the special inverse monoid $M_{Q,W}$ is determined by finite presentability of the group $K_Q$, and of the submonoid $T_W$ of $K_Q$. 
Since there are examples of
finitely generated submonoids of one-relator groups that are not finitely presented, this leads to the following result.  

\begin{thm}\label{thm_rightUnitsNotFP}
There exists a one-relator special inverse monoid $M = \Ipres{A}{r=1}$ with 
the following properties:
\begin{itemize}
\item the group of units $G$ of $M$ is finitely presented, and
\item the submonoid of right units $R$ of $M$ is finitely generated but not finitely presented.      
\end{itemize}
Hence there is no finitely presented monoid $N$ such that 
$R \cong G \ast N$.   
In particular, there is no free monoid $F$ such that $R \cong G \ast F$.   
Throughout, $M$ can be chosen to be E-unitary.  
\end{thm} 
\begin{proof}
Let $A = \{a_1, \ldots, a_n\}$, 
let $Q= \varnothing$ and let $W = \{w_1, \ldots, w_k\}$ be a finite subset of $A^*$ such that the submonoid of $A^*$ generated by $W$ is not finitely presented
(e.g. \cite[Example 4.5]{campbell96}). Then we have 
\[
M_{Q,W} = \Ipres{A,t}
{
e(a_1, \ldots, a_n,
tw_1t^{-1}, \ldots, tw_kt^{-1},
a_1^{-1}, \ldots, a_n^{-1}
) =1
}.
\]
The maximal group image of $M_{Q,W}$ is the free group $G_Q=F_{A \cup \{ t \}}$, and $K_Q = F_A$. The submonoid $T_W$ of $K_Q$ generated by $W$ is not finitely presented by choice of $W$, and hence by Theorem~\ref{thm_GeneralConstructionOneSided}\ref{it:constv} the submonoid of right units $R(M_{Q,W})$ is not finitely presented. 
The submonoid of right units $R(M_{Q,W})$ is finitely generated by 
Theorem~\ref{thm_GeneralConstructionOneSided}\ref{it:constii}.
Since $M_{Q,W}$ is E-unitary and the maximal group image is free it follows that the group of units of $M_{Q,W}$ is a finitely generated free group and in particular is finitely presented. 
Since the group of units $G$ is finitely  presented it follows that for any finitely presented monoid $N$ the free product $G \ast N$ is also finitely presented. 
As $R$ is not finitely presented we conclude that      
there is no finitely presented monoid $N$ such that
$R \cong G \ast N$.
For the last statement, if $R \cong G \ast F$ then since $R$ is finitely generated it would follow that $F$ is finitely generated and hence finitely presented (since $F$ is free). But since $G$ is finitely presented this contradicts the fact that $R \cong G \ast F$ is not finitely presented.
\end{proof}

\begin{remark}
If we modify the above construction slightly and instead simply take the inverse monoid 
\[
\Ipres{A,t}
{e(tw_1t^{-1}, \ldots, tw_kt^{-1}) 
=1
}
\]
then again the right units will not be finitely presented, but the group of units in this case will be the trivial group.
To prove the latter statement it is sufficient to observe that the elements $tw_lt^{-1}$ (with $w_l$ non-empty)  are not invertible in
the submonoid $tT_Wt^{-1}$ of $G_{Q,W}=F_{A \cup \{ t \}}$. 
To see this note that  in the proof of Theorem~\ref{thm_rightUnitsNotFP} we took $W = \{w_1, \ldots, w_k\}$ be a certain finite subset of the free monoid $A^*$. Hence the submonoid $T_W$ of $K_Q = F_A$ generated by $W$ must also be contained in $A^*$ and hence must have trivial group of units, since $T_W \subseteq A^* \subseteq F_A$ and $A^*$ has trivial group of units. Since the submonoid $tT_Wt^{-1}$ of $G_{Q,W}=F_{A \cup \{ t \}}$ is isomorphic to the monoid $T_W$ it follows that $tT_Wt^{-1}$ also has trivial group of units.    
\end{remark}

\begin{remark}
The easy example $M = \Ipres{a,b}{aba^{-1}b^{-1}=1}$ shows that even in cases where the group of units and the monoid of right units are both finitely presented, it is not the case that the monoid of right units decomposes as the free product of the groups of units and a free (or free inverse) monoid. 
Indeed, in this example it can be shown that the group of units is trivial, while the submonoid of right units is 
isomorphic to the free 
commutative monoid of rank $2$. 

Indeed, 
since $aba^{-1}b^{-1}$ is cyclically reduced it follows that the monoid $M$ is E-unitary, hence the submonoid of right units $R$ of $M$ 
is isomorphic to the submonoid of the group 
$G = \Gpres{a,b}{aba^{-1}b^{-1}=1}$ 
generated by the prefixes $a$, $ab$, and $aba^{-1}=b$. This is clearly equal to 
the submonoid of this group generated by $\{a,b\}$  
which the free commutative monoid of rank 2. 
The group of units of the free commutative monoid of rank 2 is easily seen to be the trivial group, from which it follows that $M$ has trivial group of units, as claimed. 
\end{remark}

\subsection{A finitely presented special inverse monoid whose group of units is not finitely presented}

Another key result due to Makanin (see Theorem~\ref{thm:SpecialMakanin} above) is that the group of units of a finitely presented special monoid is a finitely presented group, with the same number of defining relations. 
Here we show  that the analogous result does not hold for finitely presented special inverse monoids.

\begin{thm}\label{thm_FP}
There exists an inverse monoid $M$ defined by a finite special inverse monoid presentation 
$M = \Ipres{A}{r_i = 1 \; (i \in I)}$
such that $M$ has non-finitely-presented group of units $G$. 
Moreover, such an example exists where all $r_i$ belong to $A^+$, and so are all cyclically reduced, and $M$ is E-unitary. 
\end{thm}
\begin{proof}
It is well known that there exist
 finitely presented groups which contain finitely generated subgroups that are not finitely presented; 
 for instance such examples can already be found in the direct product of two free groups
 \cite{grunewald78}. 
In the construction described in Section~\ref{sec:Construction},
choose $Q = \{r_i \::\: i \in I \}$ and $W = \{w_j \::\:  j \in J \}$ such that $I$ and $J$ are finite, 
$W=W^{-1}$, and such that the subgroup $H_W$ of $K_Q = \Gpres{A}{r_i=1 \; (i \in I)}$ generated by $W$ is not finitely presented.   
Then by Theorem~\ref{thm_GeneralConstructionOneSided} \ref{it:constvi} , \ref{it:consti}, \ref{it:constiv} it follows that $M_{Q,W}$ is an E-unitary finitely presented special inverse monoid with non-finitely presented group of units $U(M_{Q,W})$.   
Here $U(M_{Q,W}) \cong K_Q \ast H_W$ is not finitely presented since $H_W$ is not finitely presented and $H_W$ is a retract of $K_Q \ast H_W$.    
For the last part of the statement, we  start from the presentation \eqref{eq:MQWalt} for $M_{Q,W}$ 
\begin{multline*}
M_{Q,W} = \Ipres{A, t}{ 
r_i=1 \; (i \in I), 
\; a a^{-1} = 1, 
\; a^{-1} a = 1 \; (a \in A), \\ 
\; t w_j t^{-1} t w_j^{-1} t^{-1} = 1, \; 
t w_j^{-1} t^{-1} t w_j t^{-1} = 1 \; 
(j \in J)  
}. 
\end{multline*}
We the apply a sequence of Tietze transformations on this presentation: we introduce new generators
$t'$ and $A'=\{a'\::\: a\in A\}$, representing $t^{-1}$ and $\{ a^{-1}\::\: a\in A\}$, and then replace all occurrences of the inverses by these new generators, to obtain the following presentation for $M_{Q,W}$: 
\begin{multline*}
\Ipres{A, A', t, t'}{ 
\overline{r_i}=1 \; (i \in I), 
\; a a' = 1, 
\; a' a = 1 \; (a \in A),
\; t t' = 1, \\
\; t \overline{w_j} t' t \overline{w_j^{-1}} t' = 1, \; 
t \overline{w_j^{-1}} t' t \overline{w_j} t' = 1 \; 
(j \in J)  
},
\end{multline*}
where $\overline{v}$ is the word obtained from $v$ by replacing each $a^{-1}$ by $a'$ for $a \in A$. Then this is a finite special inverse presentation for the same monoid $M_{Q,W}$ but with the property that all the defining relators are now strictly positive words. 
\end{proof} 

The examples given by this theorem show just how far the Makanin's result given in Theorem \ref{thm:SpecialMakanin} is from being true for special inverse monoid presentations. 

As explained in the proof of Theorem~\ref{thm_FP} there are well-known examples of finitely presented groups with finitely generated non-finitely presented subgroups, and the parent group can even be chosen to be a very easy group, namely the direct product of two free groups. 
The simple nature of these examples makes the process of writing down concrete examples of inverse monoids satisfying the hypothesis of Theorem~\ref{thm_FP} quite straightforward. We do just this in the next example. 

\begin{example}\label{thm:anti-Mak} 
We shall now use the result from \cite{grunewald78} to write down a concrete example of a finitely presented special inverse monoid with non-finitely presented group of units. 
Let $K_Q$ be the direct product $FG(c_1,c_2) \times FG(d_1,d_2)$ of two free groups of rank two which has the presentation 
\[
\Gpres{c_1, c_2, d_1, d_2}{c_id_jc_i^{-1}d_j^{-1} = 1, \; 
(i,j \in \{1,2\})}. 
\] 
Let $\phi: FG(c_1,c_2) \rightarrow FG(t)$ be the surjective homomorphism mapping $c_1 \mapsto t$ and $c_2 \mapsto 1$, and   
let $\psi: FG(d_1,d_2) \rightarrow FG(t)$ be the surjective homomorphism mapping $d_1 \mapsto t$ and $d_2 \mapsto 1$.   
Let 
\[
H = \{ (u,v): \phi(u) =\psi( v) \} \leq K_Q. 
\]
Then it follows from 
Grunewald's result \cite{grunewald78} that $H$ is a finitely generated but non-finitely presented subgroup of $K_Q$.   
In particular a finite group generating set for $H$ is given by the set
$\{(c_1,d_1), (c_2,1), (1,d_2)\}$.  
In terms of the presentation above these  generators are represented by the words $c_1d_1$, $c_2$ and $d_2$ respectively. 
So if we set 
\[
W =
\{ c_1d_1, c_2, d_2 \} \cup \{ d_1^{-1}c_1^{-1}, c_2^{-1}, d_2^{-1} \}  
\]        
then $W = W^{-1}$, and $K_Q$ and $H_W$ satisfy the hypotheses in the proof of Theorem~\ref{thm_FP}.    

Using this example, the argument in the proof of 
Theorem~\ref{thm_FP} then shows that if $M$ is the 
finitely presented special inverse monoid defined by the following finite presentation  
\begin{align*}
\mathrm{Inv}\lb c_1, c_2, d_1 d_2, t, C_1, C_2, D_1, D_2, T \mid 
& c_i C_i = 1, \; C_i c_i = 1, & &  (i \in \{1,2\}) \\ 
& d_i D_i = 1, \; D_i d_i = 1 & & (i \in \{1,2\}) \\
& tT=1, \\
& c_id_jC_iD_j=1, & & (i, j \in \{1,2\}), \\
& t c_2 T t C_2 T = 1, 
 & & t C_2 T t c_2 T = 1, \\
& t d_2 T t D_2 T = 1, 
& &  t D_2 T t d_2 T = 1, \\
& t c_1 d_1 T t D_1 C_1 T = 1, 
& &  t D_1 C_1 T t c_1 d_1 T = 1 \rb
  \end{align*}
then the group of units of $M$ is not finitely presented.  
Note that all the relators in this finite presentation are positive words.  
\end{example}

\subsection{Group of units finitely presented implies coherence}

Recall that a finitely presented group $G$ is said to be \emph{coherent} if every finitely generated subgroup of $G$ is finitely presented. 
A well-known question of Baumslag 
\cite[page 76]{Baumslag73}
asks whether every one-relator group is coherent. 
This question remains open in general but it has recently been shown by Louder and Wilton \cite{LW2018}, and independently by Wise \cite{Wise20}, that all one-relator groups with torsion are coherent. 

Given that above we have shown in Theorem~\ref{thm:notOneRelUnits} that the group of units of a special one-relator inverse monoid need not be a one-relator group, and we have shown in Theorem~\ref{thm_FP} that there are finitely presented special inverse monoids with non-finitely presented groups of units, it is natural to ask whether the group of units of a  special one-relator inverse monoid must be finitely presented. 

This question remains open, but in the subsection we present some results which show a close connection between this question and Baumslag's  open problem.  

\begin{thm}\label{thm_coho}
If all one-relator special inverse monoids 
$M = \Ipres{A}{r = 1}$ have finitely presented groups of units then all one-relator groups are coherent. 
In fact, if all E-unitary one-relator special inverse monoids $M = \Ipres{A}{r = 1}$ have finitely presented groups of units then all one-relator groups are coherent. 
\end{thm}
\begin{proof}
Suppose there is a one-relator group 
$G = \Gpres{A}{r=1}$ 
which is not coherent. Set $K_Q = G$ and let $W=A^{-1}$ be a finite subset of $F_A$ such that the subgroup $H_W$ of $G$ generated by $W$ is not finitely presented. 
Then  Theorem~\ref{thm_GeneralConstructionOneSided}  \ref{it:constvi}, \ref{it:consti}, \ref{it:constiv} imply that $M_{Q,W}$ is a one-relator E-unitary special inverse monoid with non-finitely presented group of units. 
\end{proof}

Since the group of units of an E-unitary special inverse monoid embeds in the maximal group image, we obtain both directions in that case.  

\begin{thm}\label{thm:coher2}
The following are equivalent: 
\begin{enumerate}
\item All one relator groups are coherent. 
\item Every E-unitary one-relator special inverse monoid $M = \Ipres{A}{r = 1}$ has a finitely presented group of units. 
\end{enumerate}
\end{thm}
\begin{proof}
(ii) $\Rightarrow$ (i) This follows from Theorem~\ref{thm_coho}.

(i) $\Rightarrow$ (ii) Suppose all one-relator groups are coherent. 
Then 
by Theorem~\ref{IMM:gens}
given any E-unitary one-relator special inverse monoid $M = \Ipres{A}{r = 1}$ its group of units $U(M)$ is isomorphic to a 
finitely generated 
subgroup of 
the maximal group image $G = \Gpres{A}{r=1}$ which is coherent, hence $U(M)$ is finitely presented.  
\end{proof} 

Combining the argument of  
(i) $\Rightarrow$ (ii) from above
with the  
recent result of Louder and Wilton \cite{LW2018} and independently Wise \cite{Wise20} that one-relator groups with torsion are coherent, we obtain the following positive result. 

\begin{thm}
Let $M = \Ipres{A}{r^m=1}$ where $m>1$, $r \in \overline{A}^*$ and $M$ is E-unitary (in particular this is true if $r$ is a cyclically reduced word).   
Then the group of units $G$ of $M$ is finitely presented. 
\end{thm}

\section{Concluding remarks and open problems}

In this section we list some open problems which naturally arise from the work done in this article.  

\begin{question}\label{question_1}
Is the group of units of a one-relator inverse monoid $\Ipres{A}{r=1}$ finitely presented?
How about a torsion one-relator inverse monoid $\Ipres{A}{r^k=1}$ ($k>1$)?
\end{question}

\begin{question}
Do the minimal invertible pieces for $\Ipres{A}{r^k=1}$ and
$\Ipres{A}{r=1}$ coincide?
\end{question}

The analogous statement for monoid presentations is true as a consequence of Adjan overlap algorithm.
For inverse monoids, it is certainly true that the minimal pieces of $\Ipres{A}{r^k=1}$ are invertible in
$\Ipres{A}{r=1}$, but it is not clear that the converse is true.

\begin{question}
Is the group of units of a one-relator inverse monoid $\Ipres{A}{r=1}$ embeddable into a one-relator group?
\end{question}

Related to this we have the following question. 

\begin{question}
Does the group of units of a one-relator inverse monoid $\Ipres{A}{r=1}$ have a soluble word problem?
\end{question}

\begin{question}
Is every finitely generated subgroup of a one-relator group the group of units of some one-relator inverse monoid?
\end{question}

\begin{question}
Is every finitely generated, recursively presented group the group of units of a finitely presented special inverse monoid?
\end{question}

\begin{question}
The special one-relator inverse monoid counterexamples in  
Section~\ref{sec:applications}
are all of the form 
$\Ipres{A}{r=1}$ where $r$ is not a reduced word.   
If we add the hypothesis that $r$ is reduced, or cyclically reduced, is it still possible to construct counterexamples? For example, it is still open whether for every cyclically reduced word $r$ the group of units of    
$\Ipres{A}{r=1}$ 
is a one-relator group (although we suspect that it is not true). 
\end{question}

\textbf{Added in proof.} 
It has recently been announced in \cite{JZLinton23} that all one-relator groups are coherent. Equivalently, by Theorem~\ref{thm:coher2} above, that result shows that every E-unitary one-relator special inverse monoid has a finitely presented group of units. The corresponding question in the non E-unitary case remains open; see Question~\ref{question_1} above.  

\smallskip

\textbf{Competing interests declaration:} The authors declare none.

\end{document}